\documentclass[10pt,a4paper,
]{article}
\usepackage[
	a4paper, scale=0.87,
]{geometry}
 
\usepackage{amsmath}
\usepackage{amsfonts}
\usepackage{amsthm}
\usepackage{amssymb}
\usepackage{latexsym}
\usepackage{fixmath}% for slanted capital Greek letters: Gamma Delta Theta Lambda Xi Pi Sigma Upsilon Phi Psi Omega
\usepackage{mathdots}

\usepackage{xcolor}
\usepackage{graphicx}
\usepackage{subfig}

\usepackage{siunitx}

\usepackage{hyperref}
\hypersetup{colorlinks}  
\usepackage[capitalize]{cleveref}
\crefname{assumption}{assumption}{assumptions}
\Crefname{figure}{Figure}{Figures}
\crefname{equation}{}{}
\Crefname{equation}{Eq.}{Eqs.}

\usepackage[shortlabels]{enumitem}
\setlist[enumerate,2]{label=(\alph*),ref=\theenumi(\alph*)}
\setlist[enumerate,3]{label=\roman*.,ref=\theenumii\roman*}

\usepackage{algorithm,algorithmic}

\usepackage{mathtools}
\providecommand\given{\:\vert\:}
\DeclarePairedDelimiterXPP\set[1]{}{\{}{\}}{}{
\renewcommand\given{\nonscript\:\delimsize\vert\nonscript\:\mathopen{}}
#1}

\DeclarePairedDelimiter\N{\|}{\|}

\def\wtd{\widetilde}
\def\what{\widehat}

\def\ee{\mathrm{e}}
\DeclareMathOperator{\diff}{d\!}

\DeclareMathOperator{\T}{T}

\DeclareMathOperator{\range}{\mathcal{R}}

\newtheorem{theorem}{Theorem}[section]
\newtheorem{lemma}{Lemma}[section]

\theoremstyle{definition}
\newtheorem{definition}{Definition}[section]

\numberwithin{equation}{section}
\numberwithin{figure}{section}
\numberwithin{table}{section}
\allowdisplaybreaks

\renewcommand\*{\discretionary{\,\text{$\cdot$}}{}{}}
\newcommand\clue[2]{\stackrel{\makebox[0pt][c]{\scriptsize #1}}{#2}\;}
\DeclareMathOperator{\toepL}{\mathcal{L}}

\DeclareMathOperator{\opE}{\mathrm{E}}
\newcommand\E[2][*]{\opE\set#1{#2}}
\usepackage{amsbsy}
\newcommand\bs[1]{\boldsymbol{ #1}}
\usepackage{mathrsfs}
\newcommand\op[1]{\mathscr{ #1}}
\newcommand\restrict{\vert}

\newcommand\liang[1]{ #1 }%{{\color{magenta} #1}}
\def\shrink{}%{\!\!\!}

\title{Stochastic algebraic Riccati equations are almost as easy as deterministic ones \liang{theoretically}
}

\author{
	Zhen-Chen Guo\thanks{Department of Mathematics, Nanjing University, Nanjing 210093, China. (\texttt{e-mail: guozhenchen@nju.edu.cn}). 
Supported in part by NSFC-11901290. %and Fundamental Research Funds for the Central Universities.
} \and 
Xin Liang\thanks{Corresponding author. Yau Mathematical Sciences Center, Tsinghua University, Beijing 100084, China, and 
Yanqi Lake Beijing Institute of Mathematical Sciences and Applications, Beijing 101408, China.
(\texttt{e-mail: liangxinslm@tsinghua.edu.cn}).
Supported in part by NSFC-11901340.
} 
}

\date{\today}

\begin{document}
\maketitle

\begin{abstract}
	Stochastic algebraic Riccati equations, \liang{also known as} rational algebraic Riccati equations, arising in linear-quadratic optimal control for stochastic linear time-invariant systems, were considered to be not easy to solve. The-state-of-art numerical methods most rely on differentiability or continuity, such as Newton-type method, LMI method, or homotopy method.
In this paper, we will build a novel theoretical framework and reveal the intrinsic algebraic structure appearing in this kind of algebraic Riccati equations. This structure guarantees that to solve them is almost as easy as to solve deterministic/classical ones, which will shed light on the theoretical analysis and numerical algorithm design for this topic.
\end{abstract}



\section{Introduction}\label{sec:introduction}
Algebraic Riccati equations (AREs) arise in various models related to control theory, especially in linear-quadratic optimal control design.
The deterministic/classical ones are considered for the deterministic linear time-invariant systems, including discrete-time algebraic Riccati equations (DAREs)
\[
	X=A^{\T}XA+Q-(A^{\T}XB+L)(R+B^{\T}XB)^{-1}(B^{\T}XA+L^{\T}),
\]
and continuous-time algebraic Riccati equations (CAREs)
\[
	A^{\T}X+XA+Q-(XB+L)R^{-1}(B^{\T}X+L^{\T})=0.
\]
During many years, people have developed rich theoretical results and numerical methods for the DAREs and CAREs. 
Readers are referred to \cite{mehrmann1991automomous,lancasterR1995algebraic,ionescuOW1999generalized,biniIM2012numerical,huangLL2018structurepreserving,bennerBKS2020numerical} to obtain an overview for both theories and algorithms.
In comparison, the stochastic/rational ones are considered for the stochastic linear time-invariant systems, including stochastic discrete-time algebraic Riccati equations (SDAREs)
\begin{equation}\label{eq:sdare}
	\begin{multlined}[b]
		X=A_{\liang{0}}^{\T}XA_{\liang{0}} + \sum_{i=1}^{\liang{r-1}} A_i^{\T}XA_i+Q
	\\\qquad\qquad\qquad
		-(A_{\liang{0}}^{\T}XB_{\liang{0}}+\sum_{i=1}^{\liang{r-1}} A_i^{\T}XB_i+L)
	(B_{\liang{0}}^{\T}XB_{\liang{0}}+\sum_{i=1}^{\liang{r-1}} B_i^{\T}XB_i+R)^{-1}(B_{\liang{0}}^{\T}XA_{\liang{0}}+\sum_{i=1}^{\liang{r-1}} B_i^{\T}XA_i+L^{\T}),
	\end{multlined}
\end{equation}
and stochastic continuous-time algebraic Riccati equations (SCAREs)
\begin{equation}\label{eq:scare}
	\begin{multlined}[t]
		A_{\liang{0}}^{\T}X+XA_{\liang{0}} + \sum_{i=1}^{\liang{r-1}} A_i^{\T}XA_i+Q
		-(XB_{\liang{0}}+\sum_{i=1}^{\liang{r-1}} A_i^{\T}XB_i+L)
		%\cdot\\\qquad\qquad\qquad
		(\sum_{i=1}^{\liang{r-1}} B_i^{\T}XB_i+R)^{-1}(B_{\liang{0}}^{\T}X+\sum_{i=1}^{\liang{r-1}} B_i^{\T}XA_i+L^{\T})=0.
	\end{multlined}
\end{equation}
Here $\liang{r-1}$ is the number of stochastic processes involved in the stochastic systems dealt with,
and it is easy to check that for the case $r=\liang{1}$ SDAREs and SCAREs degenerate to DAREs and CAREs respectively.
Due to the complicated forms, one may recognize it would be much more difficult to analyze their properties and obtain their solutions.
There are still literature, e.g., \cite{damm2004rational,draganMS2010mathematical,draganMS2013mathematical}, discussing the stochastic linear systems and the induced stochastic AREs.

As we can see, the stochastic AREs are still algebraic, and it is quite natural to ask whether algebraic methods could be developed to solve them.
However, limited by lack of clear algebraic structures, to the best of the authors' knowledge, nearly all of the existing algorithms are based on the differentiability or continuity of the equations, such as Newton's method \cite{damm2004rational,dammH2001newtons}, modified Newton's method \cite{guo2001iterative,ivanov2007iterations,chuLLW2011modified}, Lyapunov/Stein iterations \cite{fanWC2016smith,ivanov2007properties,takahashiKSS2009numerical}, comparison theorem based method \cite{freilingH2003properties,freilingH2004class}, LMI's (linear matrix inequality) method \cite{ramiZ2000linear,iidukaY2012computational}, and homotopy method \cite{zhangFCW2015homotopy}.

The key to the problem is the algebraic structures behind the equations.
In this paper, we will build up a simple and clear algebraic interpretation of SDAREs and SCAREs with the help of the so-called left semi-tensor product.
In the analysis we find out the Toeplitz structure and the symplectic structure appearing in the equations, and illustrate the fact that the fixed point iteration and the doubling iteration are also valid for them.
The algebraic structures found here will shed light on the theoretical analysis and numerical algorithms design, 
and strongly imply that stochastic AREs are almost as easy as deterministic ones.

The rest of the paper is organized as follows.
First some notations and a brief description of the left semi-tensor product are given immediately.
\Cref{sec:stochastic-discrete-time-algebraic-riccati-equations} and \cref{sec:stochastic-continuous-time-algebraic-riccati-equations:scare} are devoted to describe the algebraic structures in SDAREs and SCAREs respectively.
At last some concluding remarks are given in \Cref{sec:conclusion}.

\subsection{Notations}\label{ssec:notations}
In this paper,
${\mathbb R}$ is the set of all real numbers.
${\mathbb R}^{n\times m}$ is the set of all $n\times m$ real matrices,
${\mathbb R}^n={\mathbb R}^{n\times 1}$, and ${\mathbb R}={\mathbb R}^1$.
$I_n$ (or simply $I$ if its dimension is clear from the context) is the $n\times n$ identity matrix. %and $e_j$ is its $j$th column.
%$\mathbf{1}_n =\sum_{j=1}^n e_j$ (or also simply $\mathbf{1}$ if its dimension is clear from the context).
%$\diag(\alpha_1,\dots,\alpha_n)$ is a diagonal matrix whose diagonal entries are $\alpha_1,\dots,\alpha_n$.
\liang{
%Given a vector $x$,
%$x^{\T}$, and $\N{x}$ are its transpose, and norm respectively;
Given a matrix $X$,
}
$X^{\T}$, $\N{X}$, and  $\rho(X)$ are its transpose, induced norm, and spectral radius respectively.
%By $\Re \alpha$ denote the real part of a complex number $\alpha$.
Given a linear operator $\op X$,
$\op X^{*}$, $\N{\op X}$, and  $\rho(\op X)$ are its adjoint, norm, and spectral radius respectively.
For a symmetric matrix $X$,
$X\succ 0$ ($X\succeq 0$) indicates its positive (semi-)definiteness, and $X\prec 0$ ($X\preceq 0$) if $-X\succ 0$ ($-X\succeq 0$).

Some easy identities are given:
\begin{equation}\label{eq:easy}
	U(I+V^{\T}U)=(I+UV^{\T})U,\qquad
	U(I+V^{\T}U)^{-1}=(I+UV^{\T})^{-1}U.
\end{equation}
Here is the Sherman-Morrison-Woodbury formula: % (SMWF): 
\begin{equation}\label{eq:smwf}
	(M + UDV^{\T})^{-1} = M^{-1} - M^{-1} U (D^{-1} + V^{\T} M^{-1} U)^{-1} V^{\T} M^{-1}.
\end{equation}
The inverse sign in \cref{eq:easy,eq:smwf} indicates invertibility.
%SINUM%Both will be applied occasionally.  

%\section{Preliminaries}\label{sec:preliminaries}
\subsection{Left semi-tensor product}\label{ssec:left-semi-tensor-product}
The left semi-tensor product, first defined in 2001 \cite{cheng2001semitensor}, has many applications in system and control theory, such as  Boolean networks \cite{chengQ2009controllability} and electrical systems \cite{xueM2008new}.
Please seek more information in the monograph \cite{cheng2019dimensionfree}.

By $A \otimes B$ denote the Kronecker product of the matrices $A$ and $B$.
For $A\in \mathbb{R}^{m\times n}, B\in \mathbb{R}^{p\times q}$, define 
the left semi-tensor product of $A$ and $B$:
\[
	A\ltimes B:=
\begin{dcases*}
	(A\otimes I_{p/n})B & if $n\mid p$,\\ 
	A(B\otimes I_{n/p}) & if $p\mid n$. 
\end{dcases*}
\]
%SINUM%Clearly $A\ltimes B=AB$ if $n=p$;
%SINUM%$A\ltimes B= B\otimes A$ if $m=p=1$;
%SINUM%$A\ltimes B= A\otimes B$ if $n=q=1$.
This product satisfies:
\begin{itemize}
	\item $(A\ltimes B)\ltimes C=A\ltimes (B\ltimes C)$ (so the parenthesis can be omitted);
	\item $(A+ B)\ltimes C=A\ltimes C+B\ltimes C, A\ltimes (B+C)=A\ltimes B+A\ltimes C$;
	\item $(A\ltimes B)^{-1}=B^{-1}\ltimes A^{-1}$;
	\item $(A\ltimes B)^{\T}=B^{\T}\ltimes A^{\T}$;
	\item $\!\!\begin{bmatrix}
			A_{11} &\! A_{12} \\ A_{21} &\! A_{22}\\
			\end{bmatrix}\!\ltimes\!\begin{bmatrix}
			B_{11} &\! B_{12} \\ B_{21} &\! B_{22}\\
			\end{bmatrix}\!\!=\!\!\begin{bmatrix}
			A_{11}\ltimes B_{11}+A_{12}\ltimes B_{21} &\! 
			A_{11}\ltimes B_{12}+A_{12}\ltimes B_{22} \\
			A_{21}\ltimes B_{11}+A_{22}\ltimes B_{21} &\! 
			A_{21}\ltimes B_{12}+A_{22}\ltimes B_{22} \\
		\end{bmatrix}\!$.
\end{itemize}
The left semi-tensor product, which satisfies the same arithmetic laws as the classical matrix product, can be treated as the matrix product in the following sections.  
Briefly, we write $A^{\ltimes k}=\underbrace{A\ltimes A\ltimes \dots \ltimes A}_k$.

%\section{Stochastic discrete-time algebraic Riccati equations}\label{sec:stochastic-discrete-time-algebraic-riccati-equations}
\section{SDARE}\label{sec:stochastic-discrete-time-algebraic-riccati-equations}

Consider the SDARE \cref{eq:sdare}
%SINUM%Consider the stochastic discrete-time algebraic Riccati equations (SDARE):
%SINUM%\begin{equation}\label{eq:sdare}
%SINUM%	\begin{multlined}[t]
%SINUM%	X=A^{\T}XA + \sum_{i=1}^r A_i^{\T}XA_i+Q
%SINUM%	-(A^{\T}XB+\sum_{i=1}^r A_i^{\T}XB_i+L)\cdot
%SINUM%	\\\qquad\qquad\qquad(B^{\T}XB+\sum_{i=1}^r B_i^{\T}XB_i+R)^{-1}(B^{\T}XA+\sum_{i=1}^r B_i^{\T}XA_i+L^{\T}),
%SINUM%	\end{multlined}
%SINUM%\end{equation}
where $A_i, Q\in \mathbb{R}^{n\times n}$, $B_i\in \mathbb{R}^{n\times m}$, $L\in \liang{\mathbb{R}^{n\times m}}$ and $R\in \mathbb{R}^{m\times m}$ with 
$\begin{bmatrix}
	Q& L \\ L^{\T} & R
\end{bmatrix}\succeq 0$. 
It is easy to see that $X$ is a solution 
%!SINUM%$\Leftrightarrow$
if and only if
$X^{\T}$ is a solution.
In control theory, usually only symmetric solutions to \cref{eq:sdare} are needed.
Hence in the paper, we only consider the symmetric solutions.

The SDARE~\cref{eq:sdare} arises from linear time-invariant stochastic discrete-time control systems: 
\begin{equation}\label{eq:sdare-system}
	\begin{aligned}
		x_{t+1} &= A_{\liang{0}} x_{t} + B_{\liang{0}}u_{t}+\sum_{i=1}^{\liang{r-1}}(A_ix_{t}+B_iu_{t}) w_{i,t},
		\\
		z_{t} &= \liang{C_z}x_{t} + \liang{D_z}u_{t},
	\end{aligned}
\end{equation}
where $x_t,u_t,z_t$ are states, inputs, measurements, respectively,
and $\set{w_t=\begin{bmatrix}
		w_{1,t} &\!\!\cdots &\!\! w_{\liang{r-1},t}
\end{bmatrix}^{\T}}$ is a sequence of independent random vectors satisfying
$\E{w_t}=0,\E{w_tw_t^{\T}}=I_{\liang{r-1}}$.
%SINUM%Let $\mathcal{F}_t=\sigma(w_0,w_1,\dots,w_t)$ be the $\sigma$-algebra generated by $w_1,\dots,w_t$,
%SINUM%or equivalently, the information known in the $t$-th step.  Clearly
%SINUM%$\mathcal{F}_0\subset\mathcal{F}_1\subset \dots \subset \mathcal{F}_t\subset \mathcal{F}_{t+1}\subset\dotsb$ is a $\sigma$-algebra filtration.
Let $\set{\sigma(w_0,w_1,\dots,w_t)\given t=0,1,\dots}$ be the related $\sigma$-algebra filtration.
Write $%\bs x=\set{x_{k}}_{k\in \mathbb{N}},
\bs u=\set{u_{k}}_{k\in \mathbb{N}}$.
%\bs z=\set{z_{k}}_{k\in \mathbb{N}}$. 
Considering the  stochastic discrete-time control system \cref{eq:sdare-system}, the goal is to minimize the cost functional with respect to $\bs u$ when $x_0$ is given: 
\begin{equation}\label{eq:cost-function}
	J(x_0,\bs u) = \E{\sum_{t=0}^\infty \begin{bmatrix}
			x_{t} \\ u_{t}
			\end{bmatrix}^{\T}\begin{bmatrix}
			Q & L \\ L^{\T} & R
			\end{bmatrix}\begin{bmatrix}
			x_{t} \\ u_{t}
	\end{bmatrix}}.%\given \mathcal{F}_{0}
\end{equation}
%Here $\mathcal{F}_{-1}$ means that no condition is given.

Assume the following conditions hold throughout this section:
\begin{enumerate}[(\mbox{D}1)]
	\item \label{item:item1} $R\succ 0$; 
	\item \label{item:item2} the pair $(\set{A_i}_{i=0}^{\liang{r-1}}, \set{B_i}_{i=0}^{\liang{r-1}})$ is stabilizable, namely there exists $F\in \mathbb{R}^{m\times n}$ such that the linear operator 
		$
		\op S_F\colon \liang{\mathbb{R}^{n\times n}\to \mathbb{R}^{n\times n}}, S\mapsto\\ \begin{bmatrix}
				A_{\liang{0}}+B_{\liang{0}}F &\!\! A_1+B_1F&\!\! \cdots &\!\! A_{\liang{r-1}}+B_{\liang{r-1}}F
			\end{bmatrix}\*(I_r\otimes S) \begin{bmatrix}
			A_{\liang{0}}+B_{\liang{0}}F &\!\! A_1+B_1F&\!\!\cdots &\!\! A_{\liang{r-1}}+B_{\liang{r-1}}F
		\end{bmatrix}^{\T}
		$ is exponentially stable, or equivalently, 
		\[
			\rho(\op S_F)=\rho\left(\sum_{i=0}^{\liang{r-1}} (A_i+B_iF)\otimes (A_i+B_iF)\right)<1;
		\]
%SINUM%		and 
	\item \label{item:item3} the pair $(\set{A_i}_{i=0}^{\liang{r-1}}, C)$ is detectable with $C\in \mathbb{R}^{l\times n}$ satisfying $C^{\T}C=Q-LR^{-1}L^{\T}$, that is, $(\set{A_i^{\T}}_{i=0}^{\liang{r-1}}, \set{C_i^{\T}}_{i=0}^{\liang{r-1}})$ is stabilizable for $C_0=C$ and $C_i=0$ for $i=1,\cdots, {\liang{r-1}}$. 
\end{enumerate}
It is known that if the assumption above holds,
then \cref{eq:sdare} has a unique %bounded
positive semi-definite stabilizing solution $X_{\star}$, see, e.g., \liang{\cite[Theorem~5.14]{draganMS2010mathematical}}.
Here, $X$ is called a stabilizing solution if $\op S_{F_X}$ is exponentially stable with 
\begin{equation}\label{eq:F}
	F_X=-(\liang{\sum_{i=0}^{r-1}} B_i^{\T}XB_i+R)^{-1}(\liang{\sum_{i=0}^{r-1}} A_i^{\T}XB_i+L)^{\T}.
\end{equation} 
In fact,  $X_{\star}$ is  a stabilizing solution 
%!SINUM%$\Leftrightarrow$
if and only if
the zero equilibrium of the closed-loop system 
\begin{align*}
	x_{t+1} &= (A_{\liang{0}} +B_{\liang{0}} F_\star)x_{t} +\sum_{i=1}^{\liang{r-1}}(A_ix_{t}+B_iF_\star x_{t}) w_{i,t}
\end{align*}
is strongly exponentially stable in the mean square \liang{\cite[Remark~5.11]{draganMS2010mathematical}}, where $F_\star=F_{X_\star}$ is as in \cref{eq:F} with $X=X_\star$.
Moreover, the cost functional \cref{eq:cost-function} has an  optimal control $u_t=F_\star x_t$. 

%SINUM%Without loss of generality, we assume $B$ has full column rank and $C$ has full row rank.

\subsection{Fixed point iteration and Toeplitz structure}\label{ssec:fixed-point-iteration}
We first compute the equivalent form of \cref{eq:sdare}. 
Define 
$
	\wtd A = \begin{bmatrix}
		A_{\liang{0}} \\ A_1 \\ \vdots \\ A_{\liang{r-1}}
	\end{bmatrix}, %\qquad \qquad 
	\wtd B = \begin{bmatrix}
		B_{\liang{0}} \\ B_1 \\ \vdots \\ B_{\liang{r-1}}
	\end{bmatrix},
$
%By the definitions of $\wtd A$ and $\wtd B$, 
then \cref{eq:sdare} is equivalent to 
\[
	\begin{multlined}[t]
		X = \wtd A^{\T} (I_{\liang{r}}\otimes X)\wtd A + Q - (\wtd A^{\T}(I_{\liang{r}}\otimes X)\wtd B+L)
		%\cdot \\\qquad\qquad\qquad
		(\wtd B^{\T}(I_{\liang{r}}\otimes X)\wtd B+R)^{-1}(\wtd A^{\T}(I_{\liang{r}}\otimes X)\wtd B+L)^{\T}.
	\end{multlined}
\]
Let $\Pi$ be the permutation satisfying $\Pi^{\T}(X\otimes I_{\liang{r}})\Pi=I_{\liang{r}}\otimes X$, and define $\liang{A} = \Pi(\wtd A-\wtd BR^{-1}L^{\T}), \liang{B}=\Pi\wtd BR^{-1/2}$. 
Noticing $C^{\T}C=Q-LR^{-1}L^{\T}$,
\cref{eq:sdare} is further equivalent to  
\begin{equation}\label{eq:sdare-new}
	%\begin{aligned}[t]
%SINUM%		X &= 
%SINUM%		\what A^{\T} (X\otimes I_{r+1})\what A + Q -LR^{-1}L^{\T}-\what A^{\T}(X\otimes I_{r+1})\what B
%SINUM%		\cdot \\&\qquad\qquad\qquad\qquad\left(\what B^{\T}(X\otimes I_{r+1})\what B+I\right)^{-1}(\what A^{\T}(X\otimes I_{r+1})\what B)^{\T}
%SINUM%		\\&= 
%SINUM%		\what A^{\T}\ltimes X\ltimes \what A+Q-LR^{-1}L^{\T}-\what A^{\T}\ltimes X\ltimes \what B(\what B^{\T}\ltimes X\ltimes \what B+I_m)^{-1}\what B^{\T}\ltimes X\ltimes \what A
	X =
		\liang{A}^{\T}\ltimes X\ltimes \liang{A}+C^{\T}C-\liang{A}^{\T}\ltimes X\ltimes \liang{B}(\liang{B}^{\T}\ltimes X\ltimes \liang{B}+I_m)^{-1}\liang{B}^{\T}\ltimes X\ltimes \liang{A}
		%\\&\clue{\cref{eq:smwf}}{=} 
		%\what A^{\T}\ltimes X \ltimes (I_{rn}+\what B\what B^{\T}\ltimes X )^{-1} \ltimes \what A+C^{\T}C
		.
	%\end{aligned}
\end{equation}
Also $F_\star $ %in \cref{eq:F}
is rewritten as 
%!SINUM%$
\begin{equation*}%\label{eq:F:ltimes}
	F_\star =-R^{-1}L^{\T}-R^{-1/2}\liang{B}^{\T}\ltimes X_{\star} \ltimes(I_{\liang{rn}}+\liang{B}\liang{B}^{\T}\ltimes X_{\star})^{-1}\ltimes \liang{A}, 
%!SINUM%$
\end{equation*}
leading to 
\begin{equation}\label{eq:close-loop-matrix}
	\wtd A + \wtd B F_\star  = \Pi^{\T}(I_{rn}+\liang{B}\liang{B}^{\T}\ltimes X_{\star})^{-1}\ltimes \liang{A}.
\end{equation}

\liang{By \cref{eq:smwf}
the} equivalent form \cref{eq:sdare-new} leads us to consider
	a standard form of SDARE:
%SINUM%as follows: 
\begin{equation}\label{eq:sdare-equivalent1}
		X =A^{\T}\ltimes X \ltimes(I_{rn}+BB^{\T}\ltimes X )^{-1}\ltimes A+C^{\T}C
		:= \op D(X),
\end{equation}
where $A\in \mathbb{R}^{rn\times n}, B\in \mathbb{R}^{rn\times m},C\in \mathbb{R}^{l\times n}$ \liang{and $\op D\colon \mathbb{R}^{n\times n}\to \mathbb{R}^{n\times n}$}.
It is clear to see that \cref{eq:sdare-equivalent1}
is exactly the same as the classical DARE except that the matrix product is replaced by the left semi-tensor product, and it
is reduced to the DARE if $r=1$. 

Encouraging by the theory of DARE, one may solve the SDARE~\cref{eq:sdare-equivalent1} by the fixed point iteration:
\begin{equation}\label{eq:fixed-point}
	\begin{aligned}
		X_0& = 0,
\qquad X_1= C^{\T}C,
%SINUM%		\\
%SINUM%		X_1 &= C^{\T}C,
		\\
		X_{t+1}& =\op D(X_t) = A^{\T}\ltimes X_t \ltimes(I_{rn}+BB^{\T}\ltimes X_t )^{-1}\ltimes A+C^{\T}C.
	\end{aligned}
\end{equation}
\Cref{thm:convergence-of-fixed-point-iteration-for-sdare} analyzes the convergence of the fixed point iteration \cref{eq:fixed-point}.
\begin{theorem}[Convergence of fixed point iteration for SDAREs]\label{thm:convergence-of-fixed-point-iteration-for-sdare}
	\begin{enumerate}
		\item 
			\label{lm:sdare:monotonic}
			The operator $\op D$ is monotonic on the set consisting of all positive semi-definite matrices with respect to the partial order ``$\,\succeq$''.
			In detail, if $Z_1\succeq0,Z_2\succeq0$, then
			$
			Z_1\succeq Z_2 \Rightarrow \op D(Z_1)\succeq\op D(Z_2).
			$
		\item 
			\label{lm:fixed-point-convergence}
			The sequence $\set{X_t}$  generated by the fixed point iteration \cref{eq:fixed-point} is monotonically nondecreasing, and converges to the unique positive semi-definite  stabilizing solution $X_\star$ of the SDARE~\cref{eq:sdare-equivalent1}.
			% where the minimum is defined with respect to the partial order ``$\,\succeq$''.
			Moreover, the sequence is either finite or monotonically increasing \liang{(i.e., for any $t$, $X_{t+1}\succeq X_t, X_{t+1}\ne X_t$)}.
		\item 
			\label{lm:convergence-rate}
			The sequence $\set{X_t}$  generated by the fixed point iteration \cref{eq:fixed-point} converges R-linearly.
			In detail, \liang{there exists $Y\in \mathbb{R}^{n\times n},Y\succ 0$ such that
			\begin{equation}\label{eq:R-linearly-convergence}
				X_t\succeq X_\star -(\op S_{F_\star}^*)^t\left( X_\star [ I_n -Y X_\star]^{-1} \right),
			\end{equation}
			which implies
				$\lim\limits_{t\to\infty}\left(\frac{\N{X_t-X_\star}}{\N{X_\star}}\right)^{1/t}\le
				\liang{\rho(\op S_{F_\star})}<1.
				$
			Here $(\op S_{F_\star}^*)^t$ is the $t$ compositions of the adjoint of the operator $\op S_{F_\star}$.}
	\end{enumerate}
\end{theorem}
\begin{proof}
	First prove \cref{lm:sdare:monotonic}.
%SINUM%
	Suppose $Z_2\succ0$ and thus $Z_2$ is nonsingular.
	Then
	\begin{align*}
		Z_1\succeq Z_2
		&\Leftrightarrow Z_1^{-1}\preceq Z_2^{-1}
		%\Leftrightarrow Z_1^{-1}+BB^{\T}\preceq Z_2^{-1}+BB^{\T}
		\\&\Leftrightarrow \left(( Z_1^{-1}\otimes I_r)+BB^{\T}\right)^{-1}\succeq \left(( Z_2^{-1}\otimes I_r)+BB^{\T}\right)^{-1}
		\\&\Leftrightarrow ( Z_1\otimes I_r)\left(I_{\liang{rn}}+BB^{\T}( Z_1\otimes I_r)\right)^{-1}\succeq ( Z_2\otimes I_r)\left(I_{\liang{rn}}+BB^{\T}( Z_2\otimes I_r)\right)^{-1}
		\\&\Rightarrow \op D(Z_1)\succeq\op D(Z_2)
		.
	\end{align*}
%SINUM%
	If $Z_2$ is singular, then $Z_2+\varepsilon I\succ0$ for any $\varepsilon>0$.
	Thus, taking limits yields
	\[
		Z_1\succeq Z_2
		\Leftrightarrow Z_1+\varepsilon I \succeq Z_2+\varepsilon I
		\Rightarrow \op D(Z_1+\varepsilon I)\succeq\op D(Z_2 + \varepsilon I)
		\Rightarrow \op D(Z_1)\succeq\op D(Z_2)
		.
	\]

	Then turn to \cref{lm:fixed-point-convergence}.
%SINUM%
	Since $X_1=C^{\T}C\succeq X_0=0$, by \cref{lm:sdare:monotonic} we have $X_2=\op D(X_1)\succeq\op D(X_0)=X_1$.
	Similarly $0=X_0\preceq X_1\preceq X_2\preceq\dots\preceq X_t\preceq\cdots$, namely the sequence $\set{X_t}$ generated by \cref{eq:fixed-point} is monotonic.
	On the other hand, let $X_{\star}\succeq 0$ be the stabilizing  solution of the SDARE~\cref{eq:sdare-equivalent1}.
	Then it follows from \cref{lm:sdare:monotonic} that $X_{\star}=\op D(X_{\star})\succeq \op D(X_0)= X_1$, and similarly $X_{\star}\succeq X_t$ for any $t$, implying that $X_{\star}$ is an upper bound of $\set{X_t}_{t=0}^{\infty}$.  
	Hence $X_t$ converges. 
	Since the limit of $X_t$ is a fixed point of \cref{eq:sdare-equivalent1}, namely a positive semi-definite solution of SDARE, by the uniqueness of the positive semi-definite solution, $X_t\to X_{\star}$. 
%SINUM%
	On the other hand, if for some $t$, $X_t=X_{t+1}=\op D(X_t)$, then $X_t$ is a fixed point, namely a positive semi-definite solution, which forces $X_t=X_\star$. In other words, the iteration terminates in finite steps.

	Finally show \cref{lm:convergence-rate}.
%SINUM%
	%\liang{
	%calculate the differentials:
	%\begin{align*}
		%\diff\op D(X)
		%&=
		%\begin{multlined}[t]
			%A^{\T}\ltimes \diff X\ltimes (I+BB^{\T}\ltimes X)^{-1}\ltimes A%+A^{\T}\ltimes X\ltimes\diff\left((I+BB^{\T}\ltimes X)^{-1}\right)\ltimes A \\
			%\\
		%\!\!-A^{\T}\ltimes X\ltimes (I+BB^{\T}\ltimes X)^{-1}\ltimes BB^{\T}\ltimes \diff X\ltimes (I+BB^{\T}\ltimes X)^{-1}\ltimes A
		%\end{multlined}
		%\\&=A^{\T}\ltimes \left[I-X\ltimes (I+BB^{\T}\ltimes X)^{-1}\ltimes BB^{\T}\right]\ltimes \diff X\ltimes A_X
	%\\&\clue{\cref{eq:smwf}}{=} 
		%A_X^{\T}\ltimes \diff X\ltimes A_X
		%=:\op A_X(\diff X)
		%,
	%\end{align*}
	%where $\op A_X(Y)=A_X^{\T}(Y\otimes I)A_X=\sum_{i=0}^{r-1}(A_i+B_iF_X)^{\T}Y(A_i+B_iF_X)$
	%and as a result $\rho(\op A_X)=\rho(\op S_{F_X})$.
	%Thus, $
		%\diff\op D^t(X) = \op A_{\op D^{t-1}(X)} \dotsm\op A_{\op D(X)}\op A_X( \diff X)
		%$, $
		%\diff\op D^t(X)\vert_{X=X_\star}=\op A_{X_\star}^t (\diff X)
	%$.
	%Then
	%\[
		%X_{t+T}-X_{\star}=\op A_{X_\star}^{t} (X_T-X_\star) + \OO(\N{X_T-X_\star}^2),
	%\]
	%and by the  Gel'fand Theorem,
	%\[
		%\lim_{t\to\infty}\left(\frac{\N{X_{t+T}-X_{\star}}}{\N{X_T-X_{\star}}}\right)^{1/t}
		%\le\lim_{t\to\infty}\N{\op A_{X_\star}^{t}}^{1/t} =\rho(\op A_{X_\star})=\rho(\op S_{F_\star})
		%.
	%\]
%}
	%Finally \cref{eq:R-linearly-convergence} holds as a direct consequence.
	\liang{
	Write
	\[
		A_\star=(I_{\liang{rn}}+BB^{\T}\ltimes X_\star)^{-1}\ltimes A\in \mathbb{R}^{rn\times n}, \quad
	B_\star=(I_{\liang{rn}}+BB^{\T}\ltimes X_\star)^{-1}BB^{\T}\in \mathbb{R}^{rn\times rn}.
\]
	Note that $B_\star= B(I_{\liang{m}}+B^{\T}\ltimes X_\star\ltimes B)^{-1}B^{\T}\succeq0$ by \cref{eq:easy}.
	Then the adjoint of $\op S_{F_\star}$ is $\op S_{F_\star}^*\colon \mathbb{R}^{n\times n}\to \mathbb{R}^{n\times n},
	S\mapsto \sum_{i=0}^{r-1}(A_i+B_iF_\star)^{\T}S(A_i+B_iF_\star)=A_\star^{\T}(S\otimes I_r)A_\star=A_\star^{\T}\ltimes S\ltimes A_\star$,
	and $\rho(\op S_{F_\star}^*)=\rho(\op S_{F_\star})$.
	For $Z\in \mathbb{R}^{r^kn\times r^kn}$, define a family of operators $\op S_{\ltimes}\colon \mathbb{R}^{r^kn\times r^kn}\to \mathbb{R}^{r^{k+1}n\times r^{k+1}n},
	Z\mapsto B_\star\otimes I_{r^k} + A_\star\ltimes Z\ltimes A_\star^{\T}$. 
	%which possesses the following properties:
	It is easy to verify that 
	 $\op S_{\ltimes}(Z\otimes I_r)
	 %=B_\star\otimes I_{r^{k+1}}+(A_\star \otimes I_{r^{k+1}})(Z\otimes I_r)(A_\star \otimes I_{r^{k+1}})^T
	%=\left[B_\star\otimes I_{r^{k}}+(A_\star \otimes I_{r^{k}})Z(A_\star \otimes I_{r^{k}})^T\right]\otimes I_r
	=\op S_{\ltimes}(Z)\otimes I_r$,
	and
	$Z_1\succeq Z_2
	%\Rightarrow (A_\star \otimes I_{r^{k}})Z_1(A_\star \otimes I_{r^{k}})^T\succeq (A_\star \otimes I_{r^{k}})Z_2(A_\star \otimes I_{r^{k}})^T
	\Rightarrow \op S_{\ltimes}(Z_1)\succeq \op S_{\ltimes}(Z_2)$, namely
	$\op S_{\ltimes}$ is monotonically nondecreasing. % on the set consisting of all positive semi-definite matrices with respect to the partial order ``$\,\succeq$''

	For any $t$, write $\Delta_t:=X_\star-X_t$, and then
	\begin{align*}
		\Delta_t&=X_\star-X_t
		=\op D(X_\star)-\op D(X_{t-1})
		\\&\clue{\cref{eq:easy}}{=}A^{\T}\ltimes (I_{\liang{rn}}+X_\star \ltimes BB^{\T})^{-1}\ltimes X_\star\ltimes A
		-A^{\T}\ltimes X_{t-1}\ltimes (I_{\liang{rn}}+BB^{\T}\ltimes X_{t-1})^{-1}\ltimes A
		\\&=\underbrace{A^{\T}\ltimes (I_{\liang{rn}}+X_\star\ltimes BB^{\T})^{-1}}_{A_\star^T}\ltimes (X_\star-X_{t-1})\ltimes(I_{\liang{rn}}+BB^{\T}\ltimes X_{t-1})^{-1}\ltimes A
		\\[-\baselineskip]
		&=A_\star^{\T}\ltimes \Delta_{t-1} \ltimes (I_{\liang{rn}}+BB^{\T}\ltimes [X_\star-\Delta_{t-1}])^{-1}\overbrace{(I_{\liang{rn}}+BB^{\T}\ltimes X_\star)\ltimes A_\star}^{A}
		\\
		&=A_\star^{\T}\ltimes \Delta_{t-1}\ltimes \left(I_{\liang{rn}} - (I_{\liang{rn}}+BB^{\T}\ltimes X_\star)^{-1}BB^{\T}\ltimes \Delta_{t-1}\right)^{-1}\ltimes A_\star
		\\
		&=A_\star^{\T}\ltimes \Delta_{t-1}\ltimes \left(I_{\liang{rn}} - B_\star\ltimes \Delta_{t-1}\right)^{-1}\ltimes A_\star
		.
	\end{align*}
	Then we may obtain the relation between $\Delta_t$ and $\Delta_{t-2}$:
	\begin{align*}
		\Delta_t
		&=
		\begin{multlined}[t]
			A_\star^{\T}\ltimes \left[ A_\star^{\T}\ltimes \Delta_{t-2} \ltimes (I_{\liang{rn}}-B_\star\ltimes \Delta_{t-2})^{-1}\ltimes A_\star \right]%\\
			\ltimes \left(I_{\liang{rn}}-B_\star\ltimes\left[A_\star^{\T}\ltimes \Delta_{t-2} \ltimes (I_{\liang{rn}}-B_\star\ltimes \Delta_{t-2})^{-1}\ltimes A_\star \right]  \right)^{-1}\ltimes A_\star
		\end{multlined}	
		\\
		&\clue{\cref{eq:easy}}{=}
		%\begin{multlined}[t]
		%(A_\star^{\ltimes 2})^{\T}\ltimes \Delta_{t-2} \ltimes (I-B_\star\ltimes \Delta_{t-2})^{-1} 
		%\\\ltimes \left(I- A_\star\ltimes B_\star\ltimes A_\star^{\T}\ltimes \Delta_{t-2} \ltimes (I-B_\star\ltimes \Delta_{t-2})^{-1}  \right)^{-1}\ltimes A_\star^{\ltimes 2} 
		%\end{multlined}
		%\\
		%&=
		(A_\star^{\ltimes 2})^{\T}\ltimes \Delta_{t-2}
		\ltimes \left(I_{\liang{r^2n}}- (B_\star\ltimes \Delta_{t-2})\otimes I_r - A_\star\ltimes B_\star\ltimes A_\star^{\T}\ltimes \Delta_{t-2}   \right)^{-1}\ltimes A_\star^{\ltimes 2}
		.
	\end{align*}
	Since $(B_\star\ltimes \Delta_{t-2})\otimes I_r=(B_\star(\Delta_{t-2}\otimes I_r))\otimes I_r=(B_\star\otimes I_r)(\Delta_{t-2}\otimes I_{r^2})=(B_\star\otimes I_r)\ltimes\Delta_{t-2}$,
	\[
		\Delta_t
		=(A_\star^{\ltimes 2})^{\T}\ltimes \Delta_{t-2}
		\ltimes \left(I_{\liang{r^2n}}- \op S_{\ltimes}(B_\star)\ltimes \Delta_{t-2}\right)^{-1}\ltimes A_\star^{\ltimes 2}
		.
	\]
	Similarly, substituting $\Delta_{t-2}$ with its expression of $\Delta_{t-3}$,  we also have
	\[
		\Delta_t
		=
		(A_\star^{\ltimes 3})^{\T}\ltimes \Delta_{t-3}
		\ltimes \left(I_{\liang{r^3n}}- \op S_{\ltimes}^2(B_\star)\ltimes \Delta_{t-3}\right)^{-1}\ltimes A_\star^{\ltimes 3}
		,
	\]
	where $\op S_{\ltimes}^2=\op S_{\ltimes}\op S_{\ltimes}$ is the composition. By induction,
	\begin{align*}
		\Delta_t
		&=
			(A_\star^{\ltimes t})^{\T}\ltimes \Delta_0
			\ltimes \left(I_{\liang{r^tn}}- \op S_{\ltimes}^{t-1}(B_\star)\ltimes \Delta_0\right)^{-1}\ltimes A_\star^{\ltimes t}
		\\&=
			(A_\star^{\ltimes t})^{\T}\ltimes X_\star 
			\ltimes \left(I_{\liang{r^tn}}- \op S_{\ltimes}^t(X_0)\ltimes X_\star\right)^{-1}\ltimes A_\star^{\ltimes t}
			,
	\end{align*}
	for $X_0=0_{n\times n},\Delta_0=X_\star-X_0=X_\star, \op S_{\ltimes}(X_0)=B_\star$.

	We claim that the following holds, which will be proved soon later:
	\begin{equation}\label{eq:lm:convergence-rate:claim}
		\exists\, Y\succ 0\in \mathbb{R}^{n\times n} \quad\text{s.t.}\quad
		\op S_{\ltimes}(Y)\preceq Y\otimes I_r.
		%\op S_{\ltimes}^t(X_0)\preceq Y\otimes I_{r^t}.
	\end{equation}
	Then by the properties of $\op S_{\ltimes}$, from $X_0\prec Y$ we infer $\op S_{\ltimes}^t(X_0)\preceq \op S_{\ltimes}^t(Y)\preceq \op S_{\ltimes}^{t-1}(Y\otimes I_r)=\op S_{\ltimes}^{t-1}(Y)\otimes I_r\preceq\dots\preceq Y\otimes I_{r^t}$.
	Thus,
\begin{align*}
	\Delta_t
	&= 
	(A_\star^{\ltimes t})^{\T}\ltimes X_\star^{1/2} \ltimes 
	\left( I_{\liang{r^tn}} - X_\star^{1/2} \ltimes \op S_{\ltimes}^t(X_0) \ltimes X_\star^{1/2}\right)^{-1}\ltimes X_\star^{1/2} \ltimes A_\star^{\ltimes t}
	\\
	& \preceq (A_\star^{\ltimes t})^{\T}\ltimes X_\star^{1/2} \ltimes 
	\left( I_{\liang{n}} - X_\star^{1/2} Y  X_\star^{1/2}\right)^{-1}\ltimes X_\star^{1/2} \ltimes A_\star^{\ltimes t}
	\\
	&= (A_\star^{\ltimes t})^{\T}\ltimes X_\star
	\left( I_{\liang{n}} -Y  X_\star\right)^{-1}\ltimes A_\star^{\ltimes t}
	\\
	&= (\op S_{F_\star}^*)^t\left( X_\star 
	( I_{\liang{n}} -Y  X_\star)^{-1}\right),
\end{align*}
namely \cref{eq:R-linearly-convergence}.
	Then by the  Gel'fand Theorem,
	\[
		\lim_{t\to\infty}\left(\frac{\N{\Delta_{t}}}{\N{X_\star}}\right)^{1/t}
		\le\lim_{t\to\infty}\N{(\op S_{F_\star}^*)^{t}}^{1/t}\N{( I_{\liang{n}} -Y  X_\star)^{-1}}^{1/t}
		=\rho(\op S_{F_\star}^*)
		=\rho(\op S_{F_\star})
		.
	\]
	%Finally \cref{eq:R-linearly-convergence} holds as a direct consequence.

	Afterwards consider the claim \cref{eq:lm:convergence-rate:claim}.
	Since $X_\star$ is the unique positive semi-definite stabilizing solution,
	the linear Lyapunov operator $\op S_{F_\star}^*$ is exponentially stable, leading that the zero equilibrium of the system
	\[
		y_{t+1} = (A_0 +B_0 F_\star)^{\T}y_{t} +\sum_{i=1}^{r-1}\left((A_i+B_iF_\star)^{\T}y_{t}\right) w_{i,t}
	\]
	is strongly exponentially stable in the mean square \cite[Definition 3.1]{draganMS2010mathematical}.
	Then by \cite[Corollary 4.2]{draganMS2010mathematical}, there exists $Z\succ 0\in \mathbb{R}^{n\times n}$ satisfying 
	\[
			0\succ \begin{bmatrix}
					-Z & \!\!\!\!\!\!\!\!(\wtd A+\wtd BF_\star)^T(I_r\otimes Z) \\
					(I_r\otimes Z)(\wtd A+\wtd BF_\star) & -I_r\otimes Z\\
				\end{bmatrix}
				\clue{\cref{eq:close-loop-matrix}}{=}
				%\begin{bmatrix}
					%-Z & A_\star^T\Pi(I_r\otimes Z) \\
					%(I_r\otimes Z)\Pi^TA_\star & -I_r\otimes Z\\
				%\end{bmatrix}=
				\begin{bmatrix}
					-Z & A_\star^T(Z\otimes I_r)\Pi \\
					\Pi^T(Z\otimes I_r)A_\star & \!\!-\Pi^T(Z\otimes I_r)\Pi\\
				\end{bmatrix}.
	\]
	Thus, considering the Schur complement gives
	\begin{align*}
		0&\succ -\Pi^T(Z\otimes I_r)\Pi+\Pi^T(Z\otimes I_r)A_\star Z^{-1}A_\star^T(Z\otimes I_r)\Pi 
	\\&= -\Pi^T(Z\otimes I_r)\left[Z^{-1}\otimes I_{r}- A_\star Z^{-1}A_\star^T\right](Z\otimes I_r)\Pi 
	,
	\end{align*}
	and hence $Z^{-1}\otimes I_{r}-A_\star Z^{-1}A_\star^T\succ 0$.
	Since $B_\star\succeq0$, 
	there exists $\alpha >0$ such that $Z^{-1}\otimes I_{r}- A_\star Z^{-1}A_\star^T\succeq \alpha B_\star$. %\otimes I_r
	Then $Y=\frac{1}{\alpha}Z^{-1}$ guarantees the claim \cref{eq:lm:convergence-rate:claim}.
}
\end{proof}

Moreover, the sequence $\set{X_t}$ has a closed form, namely a non-iterative expression, as is shown in \cref{thm:fixedpoint}.
Just like what happens in DAREs \cite{guoL2022intrinsic}, the key to the form is the Toeplitz structure, defined as follows.

Given $A_{i}\in \mathbb{R}^{r^ip_1\times p_2}$ for $i=0,1,\cdots, m-1$, 
write the ${p_1\frac{r^m-1}{r-1}\times p_2\frac{r^m-1}{r-1}}$ matrix
\begin{equation*}\label{eq:toepL}
	\toepL_{r,p_1, p_2}\!\left(\!\begin{bmatrix}
			A_0\\A_1\\\vdots\\A_{m-1}
	\end{bmatrix}\!\right)\!
	=\!\!\begin{bmatrix}
		A_0     &                      &        &                             &                          &                        \\
		A_1     &\! A_0      \otimes I_r &        &                             &                          &                        \\
		A_2     &\! A_1     \otimes I_r  & \ddots &                             &                          &                        \\
		\vdots  & \ddots               & \ddots & \ddots                      &                          &                        \\
		\vdots  &                      & \ddots &\! A_1     \otimes I_{r^{m-3}} &\! A_0  \otimes I_{r^{m-2}} &                        \\
		A_{m-1} & \cdots               & \cdots &\! A_2    \otimes I_{r^{m-3}}  &\! A_1 \otimes I_{r^{m-2}}  &\! A_0\otimes I_{r^{m-1}} \\
	\end{bmatrix}%\in \mathbb{R}^{p_1\frac{r^m-1}{r-1}\times p_2\frac{r^m-1}{r-1}}
	\!.
\end{equation*}
For ease, $\toepL_{r,p_1, p_2}(A)=\toepL_{r,p_1, p_2}\left(\begin{bmatrix}
		A_0\\A_1\\\vdots\\A_{m-1}
	\end{bmatrix}\right)$ if $A=\begin{bmatrix}
	A_0\\A_1\\\vdots\\A_{m-1}
\end{bmatrix}$,
and this notation makes no confusion for the subscript $\cdot_{r,p_1, p_2}$ demonstrates how the matrix is composed.
Note that $\toepL_{r,p_1, p_2}(A)$ degenerates to a block-Toeplitz matrix in the case $r=1$.
In this paper it is called a $\ltimes$-block-Toeplitz matrix.
%It is well known that the block-Toeplitz matrices  play an important role in the discrete-time and also the continuous-time algebraic Riccati equations. In the following section we will see that the block-Toeplitz-like matrices proposed above is also important in solving the stochastic discrete-time algebraic Riccati equations.   

\begin{theorem}[Toeplitz structure in SDAREs]\label{thm:fixedpoint}
	Write 
	\begin{equation}\label{eq:noniter:X:UVT}
		V_t = \begin{bmatrix}
			C \\ C\ltimes A \\ C\ltimes A^{\ltimes 2}\\ \vdots \\ C\ltimes A^{\ltimes (t-1)}
		\end{bmatrix} 
		_{\liang{\frac{r^t-1}{r-1}l\times n}}\shrink\shrink\shrink\shrink\shrink\shrink
		,\quad
		T_t = \toepL_{r,l, m}\left(\begin{bmatrix}
				0_{l\times m}\\
				C\ltimes B \\ C\ltimes A\ltimes B \\ \vdots\\ C\ltimes A^{\ltimes (t-2)}\ltimes B
		\end{bmatrix}\right)
		_{\liang{\frac{r^t-1}{r-1}l\times \frac{r^t-1}{r-1}m}}\shrink\shrink\shrink\shrink\shrink\shrink\shrink\shrink
		%= \begin{bmatrix}
		%0         &        &         &        &    &   \\
		%CB        & 0      &         &        &    &   \\
		%CAB       & CB     & \ddots  &        &    &   \\
		%\vdots    &        & \ddots  & \ddots &    &   \\
		%\vdots    &        &         & CB     & 0  &   \\
		%CA^{t-2}B & \cdots & \cdots  & CAB    & CB & 0 \\
		%\end{bmatrix}
		,
	T_1=0. 
	\end{equation}
	Then the terms of the sequence $\set{X_t}$ generated by the fixed point iteration  \cref{eq:fixed-point} are
	\begin{equation}\label{eq:noniter:X:X}
		X_t
		= V_t^{\T}(I+T_tT_t^{\T})^{-1}V_t
		, \qquad t=1,2,\dots.
	\end{equation}

	As a result of \cref{lm:fixed-point-convergence} of \cref{thm:convergence-of-fixed-point-iteration-for-sdare} and \cref{eq:noniter:X:X},  
	the unique stabilizing solution $X_\star$ has an operator expression 
	\[
		X_{\star}=\op V^{*}(I+\op T \op T^{\T})^{-1} \op V,
		%\qquad
		\;\text{where}\;%\quad %\begin{equation*}%\label{eq:noniter:X:UVT:inf}
		\op V =\begin{bmatrix}
			C \\ C\ltimes A \\ C\ltimes A^{\ltimes 2} \\ C\ltimes A^{\ltimes 3} \\ \vdots 
		\end{bmatrix} 
		,%\quad
		\op T =\toepL_{r,l, m}\left( \begin{bmatrix}
		0 \\ C\ltimes B \\ C\ltimes A\ltimes B \\ C\ltimes A^{\ltimes 2}\ltimes B \\ \vdots\end{bmatrix} \right)
		.
	\]
\end{theorem}
\begin{proof}
	Clearly $X_1=C^{\T}C$.
%SINUM%	Substituting $X_1$ into \cref{eq:fixed-point} gives  
%SINUM%	\begin{align*}
%SINUM%		X_2
%SINUM%		&= A^{\T}\ltimes X_1\ltimes \left(I_{rn}+BB^{\T}\ltimes X_1\right)^{-1}\ltimes A+C^{\T}C
%SINUM%		\\
%SINUM%		&= A^{\T}\ltimes C^{\T}C\ltimes \left(I_{rn}+BB^{\T}\ltimes C^{\T}C\right)^{-1}\ltimes A+C^{\T}C
%SINUM%		\\
%SINUM%		&\clue{\cref{eq:easy}}{=} 
%SINUM%		A^{\T}\ltimes C^{\T}\ltimes \left(I_{rl}+C\ltimes BB^{\T}\ltimes C^{\T}\right)^{-1}\ltimes C\ltimes A + C^{\T}C
%SINUM%		\\
%SINUM%		&=
%SINUM%		\begin{bmatrix}
%SINUM%			C \\ C\ltimes A
%SINUM%		\end{bmatrix}^{\T}
%SINUM%		\begin{bmatrix}
%SINUM%			I_l &  \\
%SINUM%			& I_{rl}+C\ltimes BB^{\T}\ltimes C^{\T}
%SINUM%		\end{bmatrix}^{-1}
%SINUM%		\begin{bmatrix}
%SINUM%			C \\ C\ltimes A
%SINUM%		\end{bmatrix}
%SINUM%		\\
%SINUM%		&=
%SINUM%		\begin{bmatrix}
%SINUM%			C \\ C\ltimes A
%SINUM%		\end{bmatrix}^{\T}
%SINUM%		\left(I_{(r+1)l} + \begin{bmatrix}
%SINUM%				0_{l\times m}\\ C\ltimes B & 0_{rl\times rm}
%SINUM%				\end{bmatrix}\begin{bmatrix}
%SINUM%				0_{l\times m}\\ C\ltimes B & 0_{rl\times rm}
%SINUM%		\end{bmatrix}^{\T}\right)^{-1}
%SINUM%		\begin{bmatrix}
%SINUM%			C \\ C\ltimes A
%SINUM%		\end{bmatrix}
%SINUM%		\\
%SINUM%		&=:
%SINUM%		V_2^{\T}(I+T_2 T_2^{\T})^{-1} V_2,
%SINUM%	\end{align*}
%SINUM%	where $V_2=\begin{bmatrix}
%SINUM%		C\\ C\ltimes A
%SINUM%		\end{bmatrix}, T_2=\begin{bmatrix}
%SINUM%		0_{l\times m} \\ C\ltimes B & 0_{rl\times rm}\\
%SINUM%	\end{bmatrix}$.
%SINUM%
%SINUM%	We have already shown \cref{eq:noniter:X:X} is correct for $t=1,2$.
	Assuming \cref{eq:noniter:X:X} is correct for $t$, we are going to prove it is also correct for $t+1$.
	By the fixed  point iteration  \cref{eq:fixed-point},
	\begin{align*}
		X_{t+1}
%!SINUM%
		&=A^{\T}\ltimes X_t\ltimes \left(I_{rn}+BB^{\T}\ltimes X_t\right)^{-1}\ltimes A+C^{\T}C
%!SINUM%
		\\
		&= 
%SINUM%		\begin{multlined}[t]
		A^{\T}\ltimes V_t^{\T} \ltimes (I+T_tT_t^{\T})^{-1} \ltimes V_t \ltimes 
%SINUM%		\\
		(I_{rn}+BB^{\T}\ltimes V_t^{\T} \ltimes (I+T_tT_t^{\T})^{-1} \ltimes V_t)^{-1}\ltimes
		A+C^{\T}C
%SINUM%\end{multlined}
		\\
		&\clue{\cref{eq:easy}}{=} 
%SINUM%		\begin{multlined}[t]
			A^{\T}\ltimes V_t^{\T} \ltimes (I+T_tT_t^{\T})^{-1}\ltimes 
%SINUM%			\\
			\left(I+V_t\ltimes BB^{\T}\ltimes V_t^{\T}\ltimes (I+T_tT_t^{\T})^{-1}\right)^{-1}\ltimes V_t\ltimes 
			A + C^{\T}C
%SINUM%		\end{multlined}
		\\
		&= 
		A^{\T}\ltimes V_t^{\T} \ltimes (I +  (T_tT_t^{\T})\otimes I_r + V_t\ltimes BB^{\T} \ltimes V_t^{\T} )^{-1}\ltimes V_t\ltimes A + C^{\T}C
		\\
		&=
		\begin{bmatrix}
			C \\ V_t\ltimes A
		\end{bmatrix}^{\T}
		\begin{bmatrix}
			I_l &  \\
			& I +  (T_tT_t^{\T})\otimes I_r + V_t\ltimes BB^{\T} \ltimes V_t^{\T}
		\end{bmatrix}^{-1}
		\begin{bmatrix}
			C \\ V_t\ltimes A
		\end{bmatrix}
		\\
		&=
		\begin{bmatrix}
			C \\ V_t\ltimes A
		\end{bmatrix}^{\T}
		\left(I + \begin{bmatrix}
				0\\ V_t\ltimes B &  T_t\otimes I_r
				\end{bmatrix}\begin{bmatrix}
				0\\ V_t\ltimes B &  T_t\otimes I_r
		\end{bmatrix}^{\T}\right)^{-1}
		\begin{bmatrix}
			C \\ V_t\ltimes A
		\end{bmatrix}
%SINUM%		\\
%SINUM%		&=
%SINUM%		\begin{bmatrix}
%SINUM%			C \\ C\ltimes A \\ \vdots \\ C\ltimes A^{\ltimes t}
%SINUM%		\end{bmatrix}^{\T}
%SINUM%		(I + T_{t+1}T_{t+1}^{\T})^{-1}
%SINUM%		\begin{bmatrix}
%SINUM%			C \\ C\ltimes A \\ \vdots \\ C\ltimes A^{\ltimes t}
%SINUM%		\end{bmatrix}
		\\
		&=
		V_{t+1}^{\T}(I+T_{t+1}T_{t+1}^{\T})^{-1}V_{t+1}
		.
		%\qedhere
	\end{align*}

	\liang{Once \cref{eq:noniter:X:X} is obtained, the validity of the operator expression is essentially the same as that of the DARE, see \cite{guoL2022intrinsic}.
	}
\end{proof}

Note that $T_t$ in \cref{eq:noniter:X:UVT} is a $\ltimes$-block-Toeplitz matrix.
In particular, for the case $r=1$, the structure in \cref{eq:noniter:X:X} coincides with that of the DARE \cite{guoL2022intrinsic}. %, and $X_{\star}$ gives the Toeplitz structure \note{cite}.

Based on the iterative formula \cref{eq:fixed-point} (or, the equivalently non-iterative form \cref{eq:noniter:X:X}) and the convergence result in \cref{lm:convergence-rate} of \cref{thm:convergence-of-fixed-point-iteration-for-sdare},  one can solve the SDARE \cref{eq:sdare-equivalent1} directly by fixed point iteration method, or an analogous FTA method as that for DAREs \cite{guoL2022intrinsic}.
%With consideration on different shifts to accelerate the convergence,
%more analogous methods, like extended/rational Krylov subspace method \cite{heyouniJ2009extended,druskinSZ2014adaptive}, 
%More discussion will be shown in \cref{sec:computation}.

\subsection{Symplectic structure and doubling iteration}\label{ssec:doubling-iteration}

The fixed point iteration $\set{X_{t}}$ from \cref{eq:fixed-point}, or equivalently \cref{eq:noniter:X:X}, converges to the unique positive semi-definite stabilizing solution $X_{\star}$ linearly.
As the doubling iteration is an acceleration of the fixed point iteration for DAREs and CAREs in the sense that
the doubling iteration only computes the terms $X_1, X_2,X_4, \dots, X_{2^k}, \dots $ generated by the fixed point iteration, 
we will show the same acceleration is also valid for SDAREs \cref{eq:sdare-equivalent1}. 

As the symplectic structure plays a fundamental role in the theory of doubling iteration for DAREs, the symplectic-like structure is also necessary for SDAREs, of which the related concepts are defined in the beginning.

\begin{definition}\label{def:SSF1} 
	\begin{enumerate}
		\item 
			The matrix pair  $(M,  L)$ with $M\in \mathbb{R}^{rn\times 2p_1n}, L\in\mathbb{R}^{rn\times 2p_2n}$ is called a symplectic pair with respect to the left semi-tensor product, or a \emph{$\ltimes$-symplectic pair} for short,
			if $M\ltimes J\ltimes M^{\T}=L\ltimes J\ltimes L^{\T}$,
			where $J=\begin{bmatrix}
				0 & I_n \\ -I_n & 0
			\end{bmatrix}$. %It is well known that a matrix pair  $(M,  L)$ with $M, L\in\mathbb{R}^{2n\times 2n}$ is a symplectic pair if and only if $MJM^{\T}=LJL^{\T}$. 
		\item 
			For $M\in \mathbb{R}^{(r+1)n\times 2n}, L\in\mathbb{R}^{(r+1)n\times 2rn}$,
			the $\ltimes$-symplectic pair $(M,  L)$ is called in a first standard symplectic form with respect to the left semi-tensor product under the dimension partition $(1,r)$, or a \emph{$\ltimes$-SSF1 pair} for short, if   
			$M=\begin{bmatrix}A& 0_{rn\times n}\\  H&I_n\end{bmatrix}
%!SINUM%
			%\in \mathbb{R}^{(r+1)n\times 2n}
			_{\liang{(r+1)n\times 2n}}
			$ and $L=\begin{bmatrix}I_{rn}&G\\   0_{n\times rn}&A^{\T}\end{bmatrix}
%!SINUM%
			%\in \mathbb{R}^{(r+1)n\times 2rn}
			_{\liang{(r+1)n\times 2rn}}
			$, with $G, H$ symmetric.
			%$E\in \mathbb{R}^{rn\times n}$, $F\equiv F^{\T}\in \mathbb{R}^{n\times n}$ and  $K\equiv K^{\T}\in\mathbb{R}^{rn\times rn}$. 
		\item 
			For $M\in \mathbb{R}^{(r+1)n\times 2n}, L\in\mathbb{R}^{(r+1)n\times 2rn}$,
			assuming
			\[
				\mathcal{N}(M, L)=\set*{(M', L')\given  
					\begin{aligned}[t]
						M'\in \mathbb{R}^{(r^2+1)n\times (r+1)n}, L'\in \mathbb{R}^{(r^2+1)n\times(r^2+r)n},\atop \begin{bmatrix}
							M' &\!\! L'
						\end{bmatrix}\text{has full row rank},   L'\ltimes M=M'\ltimes L
					\end{aligned}
				}\ne \emptyset,
			\]
			the action 
			$(M,  L) \to (M'\ltimes M,  L'\ltimes L)$ 
			is called a doubling transformation of $(M, L)$ with respect to the left semi-tensor product, or \emph{$\ltimes$-doubling transformation} for short, for some $(M', L') \in\mathcal{N}(M, L)$. 
	\end{enumerate}
\end{definition}

Clearly, in the case $r=1$ the $\ltimes$-symplecticity and the $\ltimes$-doubling transformation degenerate to the classical symplecticity and the doubling transformation respectively. %In fact, the left semi-tensor product turns out to be the matrices multiplication in this case.  

Now we are ready to state the parallels for SDAREs.

Following \cref{eq:sdare-equivalent1}, it is easy to see 
\begin{equation}\label{eq:sdare-matrix}
	\begin{bmatrix}
		A & 0 \\ -C^{\T}C & I_n
	\end{bmatrix}
	\ltimes 
	\begin{bmatrix}
		I_n \\ X
	\end{bmatrix}
	=
	\begin{bmatrix}
		I_{rn} & BB^{\T} \\ 0 & A^{\T}
	\end{bmatrix}
	\ltimes 
	\begin{bmatrix}
		I_n \\ X
	\end{bmatrix}
	\ltimes \left((I_{\liang{rn}}+BB^{\T} \ltimes X)^{-1}A \right).
\end{equation}
Write
\begin{equation}\label{eq:Theta-Phi}
	\Theta=\begin{bmatrix}
		A & 0 \\ -C^{\T}C & I_n
	\end{bmatrix}
	_{\liang{(r+1)n\times 2n}}
	%\in \mathbb{R}^{(r+1)n\times 2n}
	, \qquad
	\Phi=\begin{bmatrix}
		I_{rn} & BB^{\T} \\ 0 & A^{\T}
	\end{bmatrix}
	%\in \mathbb{R}^{(r+1)n\times 2rn}
	_{\liang{(r+1)n\times 2rn}},
\end{equation}
and then
$
\Theta \ltimes J \ltimes \Theta^{\T} = \!\begin{bmatrix}
	0 &\!\!\! A \\ -A^{\T} &\!\!\! 0
\end{bmatrix}\! = \Phi \ltimes J \ltimes \Phi^{\T} 
$,
namely $(\Theta,\Phi)$ is a $\ltimes$-SSF1 pair.
Let
\begin{align*}
	\Theta'
	&
	=\begin{bmatrix}
		A\ltimes (I_{\liang{rn}}+BB^{\T}\ltimes C^{\T}C)^{-1} & 0
		\\
		-A^{\T}\ltimes(I_{\liang{rn}}+C^{\T}C\ltimes BB^{\T})^{-1}\ltimes C^{\T}C & I_n
	\end{bmatrix}
	_{\liang{(r^2+1)n\times (r+1)n}},
	\qquad \\
	\Phi'
	&
	=\begin{bmatrix}
		I_{r^2n} & A\ltimes BB^{\T}\ltimes (I_{\liang{rn}}+C^{\T}C\ltimes BB^{\T})^{-1}
		\\
		0 & A^{\T}\ltimes (I_{\liang{rn}}+C^{\T}C\ltimes BB^{\T})^{-1}
	\end{bmatrix}
	_{\liang{(r^2+1)n\times (r^2+r)n}},
\end{align*}
then $\begin{bmatrix}
	\Theta' &\Phi'
\end{bmatrix}$ has full row rank, and  
%SINUM%\begin{equation}\label{eq:symplectic-pairs}
%SINUM%	\begin{multlined}
$
		\Theta'\ltimes \Phi
		=\Phi'\ltimes \Theta
		$, 
%SINUM%		\\=
%SINUM%		\begin{bmatrix}
%SINUM%			A\ltimes (I+BB^{\T}\ltimes C^{\T}C)^{-1} & 
%SINUM%			A\ltimes (I+BB^{\T}\ltimes C^{\T}C)^{-1} \ltimes BB^{\T}
%SINUM%			\\
%SINUM%			-A^{\T}\ltimes(I+C^{\T}C\ltimes BB^{\T})^{-1}\ltimes C^{\T}C & A^{\T}\ltimes  (I+C^{\T}C\ltimes BB^{\T})^{-1}
%SINUM%		\end{bmatrix}
%SINUM%		,
%SINUM%	\end{multlined}
%SINUM%\end{equation}
which implies $(\Theta', \Phi')\in \mathcal{N}(\Theta, \Phi)$, and $(\Theta,\Phi)\to (\what \Theta, \what \Phi)=(\Theta'\ltimes \Theta, \Phi'\ltimes \Phi)$ is a $\ltimes$-doubling transformation. 
Simple computations give
\begin{equation}\label{eq:doubling-ltimes}
	\begin{aligned}
		\what \Theta
		&=
		\begin{bmatrix}
			A\ltimes (I_{\liang{rn}}+BB^{\T}\ltimes C^{\T}C)^{-1}\ltimes A & 0
			\\
			-C^{\T}C-A^{\T}\ltimes (I_{\liang{rn}}+C^{\T}C\ltimes BB^{\T})^{-1}\ltimes C^{\T}C \ltimes A & I_n
		\end{bmatrix}
		=:
		\begin{bmatrix}
			\what A & 0
			\\
			-\what H & I
		\end{bmatrix}_{\liang{(r^2+1)n\times 2n}}
		,
		\\
		\what \Phi
		&=
		\begin{bmatrix}
			I_{r^2n} & (BB^{\T}\otimes I_r)+A\ltimes BB^{\T}\ltimes (I_{\liang{rn}}+C^{\T}C\ltimes BB^{\T} )^{-1}\ltimes A^{\T}
			\\
			0 & A^{\T}\ltimes (I_{\liang{rn}}+C^{\T}C\ltimes BB^{\T})^{-1}\ltimes A^{\T}
		\end{bmatrix}
		=:
		\begin{bmatrix}
			I & \what G \\ 0 & \what A^{\T}
		\end{bmatrix}_{\liang{(r^2+1)n\times 2r^2n}}
		,
	\end{aligned}
\end{equation}
where 
\begin{alignat*}{2}
	\what A &= A\ltimes (I_{\liang{rn}}+BB^{\T}\ltimes C^{\T}C)^{-1}\ltimes A &&\liang{\in \mathbb{R}^{r^2n\times n}},
	\\
	\what H &= C^{\T}C+A^{\T}\ltimes (I_{\liang{rn}}+C^{\T}C\ltimes BB^{\T})^{-1}\ltimes C^{\T}C \ltimes A &&\liang{\in \mathbb{R}^{n\times n}},
	\\
	\what G &= (BB^{\T}\otimes I_r)+A\ltimes BB^{\T}\ltimes (I_{\liang{rn}}+C^{\T}C\ltimes BB^{\T} )^{-1}\ltimes A^{\T} &&\liang{\in \mathbb{R}^{r^2n\times r^2n}}.
\end{alignat*}
Clearly, $\what \Theta$ and $\what \Phi$ possess the same structures as $\Theta$ and $\Phi$, respectively.
Without surprising, $(\what \Theta, \what \Phi)$ is also a $\ltimes$-SSF1 pair.
Hence one can pursue another $\ltimes$-doubling transformation on $(\what \Theta,\what \Phi)$, and obtain some new $\ltimes$-SSF1 pair.
Finally a series of $\ltimes$-doubling transformations can be defined to obtain a sequence of $\ltimes$-SSF1 pairs. 

Since those $\ltimes$-symplectic pairs are composed of the triples $(A, G, H)$s,
only the iterative recursions of $(A,G,H)$ are necessary in practical computations rather than the $\ltimes$-symplectic pairs $(\Theta,\Phi)$, whose details are given in \cref{thm:sdare-doubling}.   

\begin{lemma}\label{thm:sdare-doubling}
	Consider the following iterative recursions:
	\begin{subequations}\label{eq:doubling}
		\begin{alignat}{2}
			A_{k+1}&=A_k\ltimes (I_{\liang{r^{2^{k}}n}}+G_k\ltimes H_k)^{-1}\ltimes A_k &&\liang{\in \mathbb{R}^{r^{2^{k+1}}n\times n}},
				\label{eq:sda:A}
				\\
				G_{k+1}&=G_k\otimes I_{\liang{r^{2^{k}}}}\!\!+A_k\ltimes (I_{\liang{r^{2^{k}}n}}\!\!+G_k\ltimes H_k)^{-1}\ltimes G_k\ltimes A_k^{\T} &&\liang{\in \mathbb{R}^{r^{2^{k+1}}n\times r^{2^{k+1}}n}},
				\label{eq:sda:G}
				\\
				H_{k+1}&=H_k+A_k^{\T}\ltimes H_k\ltimes (I_{\liang{r^{2^{k}}n}}+G_k\ltimes H_k)^{-1}\ltimes A_k &&\liang{\in \mathbb{R}^{n\times n}},
				\label{eq:sda:H}
			\end{alignat}
	\end{subequations}
	initially with $A_0=A, G_0=BB^{\T}$ and $H_0=C^{\T}C$. 
	Let $\Theta_k=\begin{bmatrix}
		A_k & 0 \\ -H_k & I_n
		\end{bmatrix}_{\liang{(r^{2^k}+1)n\times 2n}}$ and $\Phi_k=\begin{bmatrix}
		I_{r^{2^k}n} & G_k \\ 0 & A_k^{\T}
	\end{bmatrix}_{\liang{(r^{2^k}+1)n\times 2r^{2^k}n}}$.
	Then the following statements hold:
	\begin{enumerate}
		\item \label{thm:sdare-doubling:item1}
			$(\Theta_k, \Phi_k)$ is a $\ltimes$-SSF1 pair;
		\item \label{thm:sdare-doubling:item3}
			$(\Theta_k,\Phi_k)\to(\Theta_{k+1}, \Phi_{k+1})=(\Theta'_k\ltimes \Theta_k, \Phi'_k\ltimes \Phi_k)$ is a $\ltimes$-doubling transformation, where 
			\begin{align*}
				\Theta'_k
				&
				=\begin{bmatrix}
						A_k\ltimes (I_{\liang{r^{2^k}n}}+G_k\ltimes H_k)^{-1} & 0
						\\
						-A_k^{\T}\ltimes(I_{\liang{r^{2^k}n}}+H_k\ltimes G_k)^{-1}\ltimes H_k & I_n
					\end{bmatrix}
					_{\liang{(r^{2^{k+1}}+1)n\times (r^{2^k}+1)n}},
					\qquad 
					\\
					\Phi'_k
					&
					=\begin{bmatrix}
						I_{r^{2^{k+1}}n} & A_k\ltimes G_k\ltimes (I_{\liang{r^{2^k}n}}+H_k\ltimes G_k)^{-1}
						\\
						0 & A_k^{\T}\ltimes (I_{\liang{r^{2^k}n}}+H_k\ltimes G_k)^{-1}
					\end{bmatrix}
					_{\liang{(r^{2^{k+1}}+1)n\times (r^{2^{k+1}}+r^{2^k})n}};
			\end{align*}
		\item \label{thm:sdare-doubling:item4}
			it holds for $k=0,1,2,\dots$ that 
			\begin{equation}\label{eq:doubling-ltimes-k}
				\Theta_k\ltimes \begin{bmatrix}
					I_{\liang{n}} \\ X
				\end{bmatrix} 
				=
				\Phi_k \ltimes \begin{bmatrix}
					I_{\liang{n}} \\ X
				\end{bmatrix}
				\ltimes \left((I_{\liang{rn}}+BB^{\T}\ltimes X)^{-1}A\right)^{\ltimes 2^k}.
			\end{equation}
	\end{enumerate}
\end{lemma}
\begin{proof}
	\Cref{thm:sdare-doubling:item1,thm:sdare-doubling:item3} holds by the same discussion as \cref{eq:Theta-Phi,eq:doubling-ltimes}. 
	Now we prove \cref{thm:sdare-doubling:item4} by induction.
%SINUM%
	The case $k=0$ holds by \cref{eq:sdare-matrix,eq:Theta-Phi}. 
	Suppose it holds for $k$ and consider $k+1$.
	By $\Theta'_k\ltimes \Phi_k = \Phi'_k\ltimes \Theta_k, \Theta_{k+1}=\Theta'_k\ltimes \Theta_k, \Phi_{k+1}=\Phi'_k\ltimes \Phi_k$,
	writing $A_X=(I_{\liang{rn}}+BB^{\T}\ltimes X)^{-1}A$,
we have
	\begin{align*}
		\Theta_{k+1}\ltimes \begin{bmatrix}
			I\\X
		\end{bmatrix}
		\!= \Theta'_k\ltimes \Theta_k \ltimes \begin{bmatrix}
			I \\ X
		\end{bmatrix}\!
		&=
		\Theta'_k\ltimes \Phi_k \ltimes \begin{bmatrix}
			I \\ X
		\end{bmatrix}
		\ltimes A_X^{\ltimes 2^k}
		\\&=
		\Phi'_k\ltimes \Theta_k \ltimes \begin{bmatrix}
			I \\ X
		\end{bmatrix}
		\ltimes A_X^{\ltimes 2^k}
		\\&=
		\Phi'_k\ltimes \Phi_k \ltimes \begin{bmatrix}
			I \\ X
		\end{bmatrix}
		\ltimes A_X^{\ltimes 2^{k+1}}
		\!=\Phi_{k+1}\ltimes \begin{bmatrix}
			I\\X
		\end{bmatrix}\ltimes A_X^{\ltimes 2^{k+1}},
	\end{align*}
%SINUM%	\begin{align*}
%SINUM%		\Theta_{k+1}\ltimes \begin{bmatrix}
%SINUM%			I\\X
%SINUM%		\end{bmatrix}
%SINUM%		&= \Theta'_k\ltimes \Theta_k \ltimes \begin{bmatrix}
%SINUM%			I \\ X
%SINUM%		\end{bmatrix}
%SINUM%		\\&=
%SINUM%		\Theta'_k\ltimes \Phi_k \ltimes \begin{bmatrix}
%SINUM%			I \\ X
%SINUM%		\end{bmatrix}
%SINUM%		\ltimes \left((I+BB^{\T}\ltimes X)^{-1}A\right)^{\ltimes 2^k}
%SINUM%		\\&=
%SINUM%		\Phi'_k\ltimes \Theta_k \ltimes \begin{bmatrix}
%SINUM%			I \\ X
%SINUM%		\end{bmatrix}
%SINUM%		\ltimes \left((I+BB^{\T}\ltimes X)^{-1}A\right)^{\ltimes 2^k}
%SINUM%		\\&=
%SINUM%		\Phi'_k\ltimes \Phi_k \ltimes \begin{bmatrix}
%SINUM%			I \\ X
%SINUM%		\end{bmatrix}
%SINUM%		\ltimes \left((I+BB^{\T}\ltimes X)^{-1}A\right)^{\ltimes 2^{k+1}}
%SINUM%		\\&=\Phi_{k+1}\ltimes \begin{bmatrix}
%SINUM%			I\\X
%SINUM%		\end{bmatrix}\ltimes \left((I+BB^{\T}\ltimes X)^{-1}A\right)^{\ltimes 2^{k+1}},
%SINUM%	\end{align*}
	that is, the result holds for $k+1$. 
	Then \cref{thm:sdare-doubling:item4} is a direct consequence.
\end{proof}
For the case that $r=1$, \cref{thm:sdare-doubling} degenerates into the doubling method for DAREs (see, e.g., \cite{huangLL2018structurepreserving}), where $(\Theta_k, \Phi_k)$ are symplectic pairs in the first standard form. 

Then we prove that $H_0,H_1,H_2,\dots$ is the subsequence $X_1,X_2,X_4,\dots$ of the sequence generated by the fixed point iteration \cref{eq:fixed-point}.

\begin{lemma}\label{thm:doubling}
	For $k=0,1,2,\dots$, let 
	\begin{align*}
		U_{2^k}
		&=\begin{bmatrix}
			A^{\ltimes (2^k-1)}\ltimes B &
			(A^{\ltimes (2^k-2)}\ltimes B)\otimes I_r &
			\cdots &
			%(A^{\ltimes 2}\ltimes B)\otimes I_{r^{2^k-3}}  &
			(A\ltimes B)\otimes I_{r^{2^k-2}} &
			B\otimes I_{r^{2^k-1}} 
		\end{bmatrix},
	\end{align*}
	and $V_{2^k}, T_{2^k}$ as in \cref{eq:noniter:X:UVT}.
Then it holds that 
\begin{subequations}\label{eq:ZHGA}
	\begin{align}
	A_k&=A^{\ltimes 2^k}-U_{2^k}(I+T_{2^k}^{\T}T_{2^k})^{-1}T_{2^k}^{\T}V_{2^k}
	,\label{eq:dsda:A} \\
	G_k&=U_{2^k}(I+T_{2^k}^{\T}T_{2^k})^{-1}U_{2^k}^{\T}
	, \label{eq:dsda:G}\\
	H_k &= V_{2^k}^{\T}(I+T_{2^k}T_{2^k}^{\T})^{-1}V_{2^k}
	,\label{eq:dsda:H}
\end{align}
\end{subequations}
and so $H_k=X_{2^k}$ as in \cref{eq:noniter:X:X}.
\end{lemma}

\begin{proof}
	Induction will be used to obtain \cref{eq:ZHGA}.
	The case $k=0$ is obvious.
Now assume that \cref{eq:ZHGA} holds for $k$ and observe the case $k+1$.
For ease, we omit the subscript $\cdot_{2^k}$ for $U,V,T$.
Write $W=V\ltimes U$, and then
\[
	T_{2^{k+1}}=\begin{bmatrix}
		T & 0\\ W & T\otimes I_{r^{2^k}}
	\end{bmatrix},\quad
	U_{2^{k+1}}=\begin{bmatrix}
		A^{\ltimes 2^k}\ltimes U & U\otimes I_{r^{2^k}}
	\end{bmatrix},\quad
	V_{2^{k+1}}=\begin{bmatrix}
		V \\ V\ltimes A^{\ltimes 2^k}
	\end{bmatrix}
	.
\]
Write
%SINUM%\[
$
	M=I+T^{\T}T,\; N=I+TT^{\T},\; K=M+W^{\T}\ltimes N^{-1}\ltimes W,\; L=N\otimes I_{r^{2^k}}+WM^{-1}W^{\T},
	$
%SINUM%\]
and also
	\begin{subequations}\label{eq:thm:decoupled-form:pf:inv}
		\begin{align}
			M^{-1}&
%SINUM%			=(I+T^{\T}T)^{-1}
			\;\;
			\clue{\cref{eq:smwf}}{=}I-T^{\T}(I+TT^{\T})^{-1}T=I-T^{\T}N^{-1}T,
			\label{eq:thm:decoupled-form:pf:Minv} 
			\\
			N^{-1}&
%SINUM%			=(I+TT^{\T})^{-1}
			\;\;
			\clue{\cref{eq:smwf}}{=}I-T(I+T^{\T}T)^{-1}T^{\T}=I-TM^{-1}T^{\T},
			\label{eq:thm:decoupled-form:pf:Ninv} 
			\\
			K^{-1}&
%SINUM%			=(M+W^{\T}\ltimes N^{-1}\ltimes W)^{-1}
			\;\;\clue{\cref{eq:thm:decoupled-form:pf:Ninv}}{=}\;(M+W^{\T}W-W^{\T}\ltimes TM^{-1}T^{\T}\ltimes W)^{-1},
			\label{eq:thm:decoupled-form:pf:Kinv} 
			\\
			L^{-1}&
%SINUM%			=(N\otimes I_{r^{2^k}}+WM^{-1}W^{\T})^{-1}
			\;\;\clue{\cref{eq:thm:decoupled-form:pf:Minv}}{=}\;(N\otimes I_{r^{2^k}}+WW^{\T}-WT^{\T}N^{-1}TW^{\T})^{-1}
			.
			\label{eq:thm:decoupled-form:pf:Linv} 
		\end{align}
	\end{subequations}
	Thus,
%SINUM%	\[
	$
		G_k=UM^{-1} U^{\T},
		\;
%SINUM%		\qquad
			H_k=V^{\T} N^{-1} V,
			\;
%SINUM%			\qquad
			A_k=A^{\ltimes 2^k}-UM^{-1}T^{\T} V,
			$
%SINUM%		\]
	and
	\begin{equation} \label{eq:thm:decoupled-form:pf:5}
		\begin{aligned}[b]
			(I+G_k\ltimes H_k)^{-1}
			&=( I+ U M^{-1} W^{\T}\ltimes  N^{-1} V)^{-1}
			\\&\clue{\cref{eq:smwf}}{=}I-UM^{-1} W^{\T}\ltimes N^{-1} ( I+ WM^{-1} W^{\T}\ltimes  N^{-1} )^{-1}\ltimes V
			%\\&
			=I-U M^{-1} W^{\T} L^{-1}\ltimes V
			,
		\end{aligned}
	\end{equation}
	Then, by \cref{eq:thm:decoupled-form:pf:5},
	\begin{equation}\label{eq:thm:decoupled-form:pf:6}
		\begin{aligned}[b]
			H_k\ltimes (I+G_k\ltimes H_k)^{-1}
			&=V^{\T}N^{-1}V\ltimes(I-U M^{-1} W^{\T} L^{-1}\ltimes V)
			\\&=V^{\T}N^{-1}\ltimes (I-W M^{-1} W^{\T} L^{-1} )\ltimes V
			%\\&
			=V^{\T}\ltimes L^{-1}\ltimes V.
		\end{aligned}
	\end{equation}
	Thus, 
	\begin{align*}
%SINUM%		\MoveEqLeft[0]
		(I+T_{2^{k+1}}T_{2^{k+1}}^{\T})^{-1}
%SINUM%		\\
		&=\left(I+\begin{bmatrix}
				T & 0\\ W & T\otimes I_{r^{2^k}}
				\end{bmatrix}\begin{bmatrix}
				T & 0\\ W & T\otimes I_{r^{2^k}}
		\end{bmatrix}^{\T}\right)^{-1}
		\\&= \begin{bmatrix}
			N & TW^{\T} \\ WT^{\T} & N\otimes I_{r^{2^k}}+WW^{\T}
		\end{bmatrix}^{-1}
%SINUM%		\\&=
%SINUM%		\begin{multlined}[t]
%SINUM%			\begin{bmatrix}
%SINUM%				I & -N^{-1}TW^{\T}\\ & I
%SINUM%			\end{bmatrix}
%SINUM%			\begin{bmatrix}
%SINUM%				N^{-1} & \\ & (N\otimes I_{r^{2^k}}+WW^{\T}-WT^{\T}N^{-1}TW^{\T})^{-1}
%SINUM%			\end{bmatrix}
%SINUM%			\cdot\\
%SINUM%			\begin{bmatrix}
%SINUM%				I & \\ -WT^{\T}N^{-1} & I
%SINUM%			\end{bmatrix}
%SINUM%		\end{multlined}
		\\&\clue{\cref{eq:thm:decoupled-form:pf:Linv}}{=}\;\begin{bmatrix}
			I & -N^{-1}TW^{\T}\\ & I
		\end{bmatrix}
		\begin{bmatrix}
			N^{-1} & \\ & L^{-1}
		\end{bmatrix}
		\begin{bmatrix}
			I & \\ -WT^{\T}N^{-1} & I
		\end{bmatrix}
		.
	\end{align*}
%SINUM%	Noticing that 
%SINUM%	\begin{align*}
	Note that $	
		\begin{bmatrix}
				I & \\ -WT^{\T}N^{-1} &I
			\end{bmatrix}V_{2^{k+1}}
%SINUM%			&=\begin{bmatrix}
%SINUM%				I & \\ -WT^{\T}N^{-1} & I
%SINUM%				\end{bmatrix}\begin{bmatrix}
%SINUM%				V \\ V\ltimes A^{\ltimes 2^k}
%SINUM%			\end{bmatrix}
%SINUM%			\\&
			=\begin{bmatrix}
				V \\ V\ltimes A^{\ltimes 2^k}-WT^{\T}N^{-1}V
			\end{bmatrix}
			=\begin{bmatrix}
				V \\ V\ltimes A_k
			\end{bmatrix}
%SINUM%			,
%SINUM%	\end{align*}
%SINUM%	we have
			$.  Then
	\begin{align*}
		V_{2^{k+1}}^{\T} (I+T_{2^{k+1}}T_{2^{k+1}}^{\T})^{-1} V_{2^{k+1}}
		&=\begin{bmatrix}
			V \\ V\ltimes A_k
			\end{bmatrix}^{\T}\begin{bmatrix}
			N^{-1} & \\ & L^{-1}
			\end{bmatrix}\begin{bmatrix}
			V \\ V\ltimes A_k
		\end{bmatrix}
%SINUM%		\\&=
%SINUM%		V^{\T}N^{-1}V+  (A_k^{\T}\ltimes V^{\T})L^{-1} (V\ltimes A_k)
		\\&\clue{\cref{eq:thm:decoupled-form:pf:6}}{=} H_k + A_k^{\T}\ltimes H_k\ltimes (I_n + G_k\ltimes  H_k)^{-1}\ltimes A_k
%SINUM%		\\&
		\clue{\cref{eq:sda:H}}{=}\; H_{k+1},
	\end{align*}
	which implies \cref{eq:dsda:H} holds for $k+1$.
	On the other hand,
	similarly, 
	we have
	\begin{gather*}
	%\nonumber %\label{eq:thm:decoupled-form:pf:2}
			(I+G_k\ltimes H_k)^{-1}
			=I-U K^{-1} W^{\T}\ltimes  N^{-1} V
			,
			\\
	%\nonumber %\label{eq:thm:decoupled-form:pf:3}
			(I+G_k\ltimes H_k)^{-1}\ltimes G_k
			=UK^{-1}U^{\T}
			,
			\\
		(I+T_{2^{k+1}}^{\T}T_{2^{k+1}})^{-1}
		=\begin{bmatrix}
			I & \\ -M^{-1}T^{\T}\ltimes W& I
		\end{bmatrix}
		\begin{bmatrix}
			K^{-1} & \\ &\!\!\!\!\! M^{-1}\otimes I_{r^{2^k}}
		\end{bmatrix}
		\begin{bmatrix}
			I & -W^{\T}\ltimes TM^{-1}\\  & I
		\end{bmatrix}
		,%\nonumber
		\\
		U_{2^{k+1}}\begin{bmatrix}
			I & \\ -M^{-1}T^{\T}\ltimes W& I
		\end{bmatrix}
		=\begin{bmatrix}
			A_k\ltimes U & U\otimes I_{r^{2^k}}
		\end{bmatrix}
		,%\nonumber
		\\
		U_{2^{k+1}} (I+T_{2^{k+1}}^{\T}T_{2^{k+1}})^{-1} U_{2^{k+1}}^{\T}
		= G_{k+1},%\nonumber
	\end{gather*}
	which implies \cref{eq:dsda:G} holds for $k+1$.
	Similarly, 
	\begin{align*}
		\MoveEqLeft[0] A^{\ltimes 2^{k+1}}- U_{2^{k+1}}(I+T_{2^{k+1}}^{\T}T_{2^{k+1}})^{-1}T_{2^{k+1}}^{\T} V_{2^{k+1}}
		\\&=
%SINUM%		\begin{multlined}[t]
			A^{\ltimes 2^{k+1}}\!- \!\begin{bmatrix}
				A_k\ltimes U &\!\!\!\! U\otimes I_{r^{2^k}}
				\end{bmatrix}\!\!\!\begin{bmatrix}
				K^{-1} & \\ &\!\!\!\!\!\!\!\! M^{-1}\otimes I_{r^{2^k}}
				\end{bmatrix}
				\!\!\!
%SINUM%				\cdot\\
				\begin{bmatrix}
				I &\!\!\!\! -W^{\T}\ltimes TM^{-1}\\  & I
				\end{bmatrix}\!\!\!\begin{bmatrix}
				T^{\T} &\!\!\!\!\! W^{\T} \\ &\!\!\!\!\! T^{\T}\otimes I_{r^{2^k}}
				\end{bmatrix}\!\!\!\begin{bmatrix}
				V \\ V\ltimes A^{\ltimes 2^k}
			\end{bmatrix}
%SINUM%		\end{multlined}
%SINUM%		\\&= 
%SINUM%		\begin{multlined}[t]
%SINUM%			A^{\ltimes 2^{k+1}}- \begin{bmatrix}
%SINUM%				A_k\ltimes U & U\otimes I_{r^{2^k}}
%SINUM%			\end{bmatrix}
%SINUM%			\begin{bmatrix}
%SINUM%				K^{-1} & \\ & M^{-1}\otimes I_{r^{2^k}}
%SINUM%			\end{bmatrix}
%SINUM%			\cdot\\
%SINUM%			\begin{bmatrix}
%SINUM%				T^{\T}V+W^{\T}\ltimes V\ltimes A^{\ltimes 2^k}-W^{\T}\ltimes TM^{-1}T^{\T}V\ltimes A^{\ltimes 2^k}\\
%SINUM%				T^{\T}V\ltimes A^{\ltimes 2^k} \\
%SINUM%			\end{bmatrix}
%SINUM%		\end{multlined}
%SINUM%		\\&= 
%SINUM%		A^{\ltimes 2^{k+1}}- \begin{bmatrix}
%SINUM%			A_k\ltimes U & U\otimes I_{r^{2^k}}
%SINUM%		\end{bmatrix}
%SINUM%		\begin{bmatrix}
%SINUM%			K^{-1} & \\ & M^{-1}\otimes I_{r^{2^k}} \end{bmatrix}
%SINUM%			
%SINUM%		\begin{bmatrix}
%SINUM%			T^{\T}V+W^{\T}\ltimes N^{-1}V\ltimes A^{\ltimes 2^k}\\
%SINUM%			T^{\T}V\ltimes A^{\ltimes 2^k} \\
%SINUM%		\end{bmatrix}
		\\&= 
		A^{\ltimes 2^{k+1}}- A_k\ltimes UK^{-1}(T^{\T}V+W^{\T}\ltimes N^{-1}V\ltimes A^{\ltimes 2^k})- UM^{-1}T^{\T}V\ltimes A^{\ltimes 2^k}
		\\&= 
		A_k\ltimes [A^{\ltimes 2^k}- UK^{-1}T^{\T}V-UK^{-1}W^{\T}\ltimes N^{-1}V\ltimes A^{\ltimes 2^k}]
		\\&= 
		A_k\ltimes (I-UK^{-1}W^{\T}\ltimes N^{-1}V)\ltimes[A^{\ltimes 2^k}- (I-UK^{-1}W^{\T}\ltimes N^{-1}V)^{-1}UK^{-1}T^{\T}V]
%SINUM%		\\&\clue{\cref{eq:thm:decoupled-form:pf:2}}{=} 
%SINUM%		A_k\ltimes(I+G_k\ltimes H_k)^{-1}\ltimes [A^{\ltimes 2^k}- (I+UM^{-1}W^{\T}\ltimes N^{-1}V)\ltimes UK^{-1}T^{\T}V]
%SINUM%		\\&=
%SINUM%		A_k\ltimes(I+G_k\ltimes H_k)^{-1}\ltimes [A^{\ltimes 2^k}- UM^{-1}T^{\T}V]
		\\&= A_k\ltimes (I+G_k\ltimes H_k)^{-1}\ltimes A_k
%SINUM%		\\&
		\clue{\cref{eq:sda:A}}{=}\; A_{k+1},
	\end{align*}
	which implies \cref{eq:dsda:A} holds for $k+1$.
\end{proof}

%\begin{remark}\label{rk:doubling2}
For the case $r=1$, \cref{eq:doubling-ltimes-k} coincides with the decoupled formulae of the dSDA for DAREs introduced in \cite{guoCLL2020decoupled}. 
%\end{remark}
%Regarding the sequence $\{G_k\}_{k=0}^{\infty}$, one may care about if it converges and the corresponding  limit. 
\Cref{thm:convergence-of-doubling-iteration-for-sdare} is a direct consequence of \cref{thm:convergence-of-fixed-point-iteration-for-sdare,thm:doubling}.
\begin{theorem}[Convergence of doubling iteration for SDAREs]\label{thm:convergence-of-doubling-iteration-for-sdare}
	The sequence $\set{H_k}$  generated by the doubling iteration \cref{eq:doubling} with $A_0=A,G_0=BB^{\T},H_0=C^{\T}C$ is either finite or monotonically increasing, and converges to the unique positive semi-definite  stabilizing solution $X_\star$ of the SDARE~\cref{eq:sdare-equivalent1} R-quadratically, namely
			\liang{
			\begin{equation}\label{eq:R-quadratically-convergence}
				H_k\succeq X_\star -(\op S_{F_\star}^*)^{2^k}\left( X_\star [ I_{\liang{n}} -Y X_\star]^{-1} \right),
			\end{equation}
			where $Y$ and $(\op S_{F_\star}^*)^{2^k}$ are as in \cref{eq:R-linearly-convergence},
			which implies
				$\lim\limits_{t\to\infty}\left(\frac{\N{H_k-X_\star}}{\N{X_\star}}\right)^{1/2^k}\!\!\!\!\le
				\rho(\op S_{F_\star})<1.
				$
			}
\end{theorem}

Based on the doubling iteration \cref{eq:doubling},  one can solve the SDARE \cref{eq:sdare-equivalent1} directly by doubling iteration method, or equivalently an analogous SDA method as that for DAREs \cite{anderson1978secondorder,huangLL2018structurepreserving}.
%More discussion will be shown in \cref{sec:computation}.

%\section{Stochastic continuous-time algebraic Riccati equations}\label{sec:stochastic-continuous-time-algebraic-riccati-equations:scare}
\section{SCARE}\label{sec:stochastic-continuous-time-algebraic-riccati-equations:scare}

Consider the SCARE \cref{eq:scare}
%SINUM%Consider the stochastic continuous-time algebraic Riccati equation (SCARE):
%SINUM%\begin{equation}\label{eq:scare}
%SINUM%	\begin{multlined}[t]
%SINUM%		A^{\T}X+XA + \sum_{i=1}^r A_i^{\T}XA_i+Q
%SINUM%		\\-(XB+\sum_{i=1}^r A_i^{\T}XB_i+L)(\sum_{i=1}^r B_i^{\T}XB_i+R)^{-1}(B^{\T}X+\sum_{i=1}^r B_i^{\T}XA_i+L^{\T})=0,
%SINUM%	\end{multlined}
%SINUM%\end{equation}
where $A_i, Q\in \mathbb{R}^{n\times n}$, $B_i\in \mathbb{R}^{n\times m}$, $L\in \liang{\mathbb{R}^{n\times m}}$ and $R\in \mathbb{R}^{m\times m}$ with 
$\begin{bmatrix}
	Q& L \\ L^{\T} & R
\end{bmatrix}\succeq 0$.
It is easy to see that $X$ is a solution if and only if $X^{\T}$ is a solution.
In control theory, usually only symmetric solutions to \cref{eq:scare} are needed.
Hence in the paper, we only consider the symmetric solutions.

The SCARE \cref{eq:scare} arises from the stochastic time-invariant  control system in continue-time subject to multiplicative white noise, whose dynamics is described as below:
\begin{equation}\label{eq:scare-system:scare}
	\begin{aligned}
		\diff x(t) &= A_{\liang{0}}x(t)\diff t + B_{\liang{0}}u(t)\diff t+\sum_{i=1}^{\liang{\liang{r-1}}}(A_ix(t)+B_iu(t))\diff w_{i}(t),
		\\
		z(t) &= \liang{C_z}x(t) + \liang{D_z}u(t),
	\end{aligned}
\end{equation}
in which $x(t),u(t)$ and $z(t)$ are state, input, measurement, respectively,
and $w(t)=\begin{bmatrix}
	w_{1}(t) &\cdots &w_{\liang{r-1}}(t)
\end{bmatrix}^{\T}$ is a  standard Wiener process  satisfying 
that each $w_i(t)$ is a standard Brownian motion and the $\sigma$-algebras $\sigma\left(w_i(t), t\in [t_0,\infty)\right), i=1,\dots,\liang{r-1}$ are independent \cite{draganMS2013mathematical}.
Considering the cost functional with respect to the control $u(t)$ with the given initial $x_0$: 
\begin{equation}\label{eq:cost-function:scare}
	J(t_0,x_0;u) =\E{\int_{t_0}^{\infty} 
		\begin{bmatrix}
			x_{t_0,x_0;u}(t) \\ u(t)
		\end{bmatrix}^{\T}
		\begin{bmatrix}
			Q & L \\ L^{\T} & R
		\end{bmatrix}
		\begin{bmatrix}
			x_{t_0,x_0;u}(t) \\ u(t)
	\end{bmatrix}\diff t},
\end{equation}
where $x_{t_0,x_0;u}(t)$ is the solution of the system \cref{eq:scare-system:scare} corresponding to the input $u(t)$ and having the initial $x_{t_0,x_0;u}(t_0)=x_0$, one goal in stochastic control is to minimize the cost functional \cref{eq:cost-function:scare} and compute an optimal control. Such an optimization problem is also called the first linear-quadratic optimization problem %(LQOP1)
\liang{\cite[Section~6.2]{draganMS2013mathematical}}.    

Assume the following conditions hold throughout this section:
\begin{enumerate}[(\mbox{C}1)]
	\item \label{item:item1:scare} $R\succ 0$;
%SINUM%		,
%SINUM%		$\begin{bmatrix}
%SINUM%			Q& L \\ L^{\T} & R
%SINUM%		\end{bmatrix}\succeq 0$ (or equivalently  $Q-LR^{-1}L^{\T}\succeq 0$); 
	\item \label{item:item2:scare} the pair $(\set{A_i}_{i=0}^{\liang{r-1}}, \set{B_i}_{i=0}^{\liang{r-1}})$ is stabilizable, i.e., there exists $F\in \mathbb{R}^{m\times n}$ such that the linear differential equation   
		\begin{equation*}\label{eq:continus-Lyap:scare}
			\frac{\diff}{\diff t}S(t)= \op L_F S(t):=(A_{\liang{0}}+B_{\liang{0}}F)S+S(A_{\liang{0}}+B_{\liang{0}}F)^{\T}+\sum_{i=1}^{\liang{r-1}}(A_i+B_iF)S(A_i+B_iF)^{\T}
		\end{equation*}
		is exponentially stable, or equivalently, the evolution operator $\ee^{\op L_F(t-t_0)}$ is exponentially stable with $\ee^{\op L_F t}=\sum_{k=0}^{\infty}\frac{\op L_F^kt^k}{k!}$; and  
	\item \label{item:item3:scare} the pair $(\set{A_i}_{i=0}^{\liang{r-1}}, C)$ is detectable % or uniformly observable 
		with $C^{\T}C=Q-LR^{-1}L^{\T}$, or equivalently, $(\set{A_i^{\T}}_{i=0}^{\liang{r-1}}, \set{C_i^{\T}}_{i=0}^{\liang{r-1}})$ is stabilizable with $C_0=C$ and $C_i=0$ for $i=1,\cdots, \liang{r-1}$. 
\end{enumerate}
It is known that if the assumption above holds,
then \cref{eq:scare} has a unique %bounded
positive semi-definite stabilizing solution $X_{\star}$, see, e.g., \liang{\cite[Theorem~5.6.15]{draganMS2013mathematical}}.
Here, $X$ is a stabilizing solution if the system $(A_{\liang{0}}+B_{\liang{0}}F_X, A_1+B_1F_X, \cdots, A_{\liang{r-1}}+B_{\liang{r-1}}F_X)$ is stable  with 
\begin{equation}\label{eq:F:scare}
	F_X=-(\sum_{i=1}^{\liang{r-1}} B_i^{\T}XB_i+R)^{-1}(B_{\liang{0}}^{\T}X+\sum_{i=1}^{\liang{r-1}} B_i^{\T}XA_i+L^{\T}), 
\end{equation}
or equivalently, $\op L_{F_\star}$ is exponentially stable with the associated $F_\star=F_{X_\star}$ taking the feedback control specified in \cref{eq:F:scare} with $X=X_\star$. 
In fact,  $X_{\star}$ is  a stabilizing solution if and only if the zero equilibrium of the closed-loop system 
\begin{equation}\label{eq:close-loop:scare}
	\diff x(t) = (A_{\liang{0}} + B_{\liang{0}}F_{\star})x(t)\diff t+\sum_{i=1}^{\liang{r-1}}(A_i+B_iF_{\star})x(t)\diff w_{i}(t).
\end{equation}
is strongly exponentially stable in the mean square \liang{\cite[Chapter~5]{draganMS2013mathematical}}. %where $F_\star$ is as in \cref{eq:F:scare}.
Furthermore, the cost functional \cref{eq:cost-function:scare} has an  optimal control $u(t)=F_{\star}x_{t_0,x_0}(t)$ where $x_{t_0,x_0}(t)$ is the solution to the corresponding closed-loop system \cref{eq:close-loop:scare}.

%SINUM%Without loss of generality, we assume $B$ has full column rank and $C$ has full row rank.

\subsection{Standard form and symplectic structure}\label{ssec:the-standard-form:scare}
As we have done for SDAREs, first we make an equivalent reformulation for \cref{eq:scare} for the sake of simplicity. 

Write $\wtd A=\begin{bmatrix}
	A_1 \\ \vdots\\ A_{\liang{r-1}}
	\end{bmatrix}, \wtd B=\begin{bmatrix}
	B_1 \\ \vdots \\ B_{\liang{r-1}}
\end{bmatrix}$, and then \cref{eq:scare} will be rewritten as 
\begin{multline*}
	A_{\liang{0}}^{\T}X+XA_{\liang{0}} + \wtd A^{\T}(I\otimes X)\wtd A+Q-(XB_{\liang{0}}+\wtd A^{\T}(I\otimes X)\wtd B+L)
	%\cdot\\
	(\wtd B^{\T}(I\otimes X)\wtd B+R)^{-1}(B_{\liang{0}}^{\T}X+\wtd B^{\T}(I\otimes X)\wtd A+L^{\T})=0.
\end{multline*}
Let $\Pi$ be the permutation satisfying $\Pi^{\T}(X\otimes I_{\liang{r-1}}) \Pi = I_{\liang{r-1}}\otimes X$,
and 
write $\what A=\Pi(\wtd A-\wtd BR^{-1}L^{\T}), \what B=\Pi\wtd BR^{-1/2}$.
\liang{Also write $A = A_{\liang{0}}-B_{\liang{0}}R^{-1}L^{\T}, B=B_{\liang{0}}R^{-1/2}$}. 
Noticing $C^{\T}C=Q -LR^{-1}L^{\T}$,
after some calculations \cref{eq:scare} is reformulated in the standard form of SCARE
\begin{equation}\label{eq:scare-standard-form:scare}
	\begin{multlined}[t]
		A^{\T}X+XA + C^{\T}C
		+ \what A^{\T}\ltimes X\ltimes \what A %\\
		-(XB+\what A^{\T}\ltimes X\ltimes \what B)(\what B^{\T}\ltimes X\ltimes \what B+I)^{-1}(B^{\T}X+\what B^{\T}\ltimes X\ltimes \what A) 
		=0
		,
	\end{multlined}
\end{equation}
where %$X\in \mathbb{R}^{n\times n}$ is positive semi-definite,
$A\in \mathbb{R}^{n\times n}$, $B\in \mathbb{R}^{n\times m}$, $\what A\in \mathbb{R}^{\liang{(r-1)}n\times n}$,  $\what B\in \mathbb{R}^{\liang{(r-1)}n\times m}$. %and both $C\in \mathbb{R}^{l\times n}$ and $B\in \mathbb{R}^{n\times m}$ are of full column rank. 
Also the feedback control $F_X$ and the closed-loop matrix are reformulated as 
\begin{align*}
	F_X
	&=\liang{-R^{-1}L^{\T}+R^{-1/2}\what F_X},
	\\
	\begin{bmatrix}
		A_{\liang{0}}+B_{\liang{0}}F_X
		\\
		\wtd A+\wtd BF_X
	\end{bmatrix}
	&=
	\begin{bmatrix}
		A
		\\
		\Pi^{\T}\what A
	\end{bmatrix}
	+
	\begin{bmatrix}
		B
		\\
		\Pi^{\T}\what B
	\end{bmatrix}\liang{\what F_X}
	,
\end{align*}
\liang{
where 
$\what F_X=-(\what B^{\T}\ltimes X \ltimes \what B+I)^{-1}(XB + \what A^{\T}\ltimes X \ltimes  \what B)^{\T}$ is the feedback control of the standard form \cref{eq:scare-standard-form:scare}.
}
%\subsection{Invariant subspace form}\label{ssec:invariant-subspace-form:scare}
Then \cref{eq:scare-standard-form:scare} can be rewritten as
\begin{align*}
	0&=C^{\T}C+A^{\T}X+XA + \what A^{\T} \ltimes X\ltimes  \what A+
	(XB + \what A^{\T} \ltimes X \ltimes \what B)\liang{\what F_X}
	\\&=
	\begin{bmatrix}
		C^{\T}C &\! A^{\T}
	\end{bmatrix}\!
	\begin{bmatrix}
		I_n \\ X
	\end{bmatrix}
	+\begin{bmatrix}
		I_n &\! \what A^{\T}
		\end{bmatrix}\! \begin{bmatrix}
		X & \\ &\! X\otimes I_{\liang{r-1}}
		\end{bmatrix}\!  \begin{bmatrix}
		A \\ \what A
	\end{bmatrix}
	+\begin{bmatrix}
		I_n &\! \what A^{\T}
		\end{bmatrix}\! \begin{bmatrix}
		X & \\ &\! X\otimes I_{\liang{r-1}}
		\end{bmatrix}\!   \begin{bmatrix}
		B \\ \what B
	\end{bmatrix}\liang{\what F_X}
	.
\end{align*}
Let $\wtd \Pi$ be the permutation satisfying $\begin{bmatrix}
	X & \\ & X \otimes I_{\liang{r-1}}
\end{bmatrix} = \wtd \Pi^{\T}(X\otimes I_{\liang{r}})\wtd \Pi$,
and write 
$A_F=A+B\liang{\what F_X}, \what A_F=\what A+\what B\liang{\what F_X}$.
Then \cref{eq:scare-standard-form:scare} becomes
\begin{equation}\label{eq:scare-standard-form-part1:scare}
	\begin{aligned}[t]
		\begin{bmatrix}
			C^{\T}C & A^{\T}
		\end{bmatrix}
		\ltimes
		\begin{bmatrix}
			I_n \\ X
		\end{bmatrix}
		&=-\begin{bmatrix}
			I_n & \what A^{\T}
			\end{bmatrix}\wtd \Pi^{\T}\ltimes X \ltimes \wtd \Pi \begin{bmatrix}
			A_F \\ \what A_F
		\end{bmatrix}
		%\\&
		=
		-\begin{bmatrix}
			0_{n\times \liang{r}n} & I_n & \what A^{\T}
			\end{bmatrix} \begin{bmatrix}
			\wtd \Pi^{\T}& \\ & \wtd \Pi^{\T}
			\end{bmatrix} \ltimes \begin{bmatrix}
			I_n \\ X
			\end{bmatrix}\ltimes \wtd \Pi \begin{bmatrix}
			A_F \\ \what A_F
		\end{bmatrix}
		.
	\end{aligned}
\end{equation}
Note that \cref{eq:scare} is equivalent to \cref{eq:scare-standard-form-part1:scare} and 
\begin{equation}\label{eq:scare-standard-form-part2:initial:scare}
	\begin{bmatrix}
		A_F \\ \what A_F
		\end{bmatrix}=\begin{bmatrix}
		A\\ \what A
		\end{bmatrix}-\begin{bmatrix}
		B\\\what B
	\end{bmatrix}
	(\what B^{\T}\ltimes X \ltimes \what B+I)^{-1}(B^{\T}X + \what B^{\T}\ltimes X\ltimes \what A)
	.
\end{equation}
We can somehow treat \cref{eq:scare-standard-form-part1:scare} as an invariant subspace form,
which urges us to transform \cref{eq:scare-standard-form-part2:initial:scare} into that kind.

By left-multiplying the nonsingular matrix
\begin{align*}
	\begin{bmatrix}
		I_n &  B\what B^{\T}\ltimes X\\
		&I_{\liang{(r-1)}n}+\what B\what B^{\T}\ltimes X \\
	\end{bmatrix}
	&=\begin{bmatrix}
		I_n        &0&0& B\what B^{\T} \\
		0 & I_{\liang{(r-1)}n}&0& \what B\what B^{\T} \\
		\end{bmatrix}\begin{bmatrix}
		I_n \\
		& I_{\liang{(r-1)}n}\\
		X \\
		& X\otimes I_{\liang{r-1}}\\
	\end{bmatrix}
	\\&=
	\begin{bmatrix}
		I_n        &0&0& B\what B^{\T} \\
		0 & I_{\liang{(r-1)}n}&0& \what B\what B^{\T} \\
		\end{bmatrix}\begin{bmatrix}
		\wtd \Pi^{\T} & \\ & \wtd \Pi^{\T}
		\end{bmatrix} \ltimes \begin{bmatrix}
		I\\ X\\
	\end{bmatrix}\ltimes \wtd \Pi
\end{align*}
on both sides,
\cref{eq:scare-standard-form-part2:initial:scare} is equivalent to
\begin{equation}\label{eq:scare-standard-form-part2:scare}
	\begin{aligned}[t]
		\MoveEqLeft[0]	\begin{bmatrix}
			I_n        &0&0& B\what B^{\T} \\
			0 & I_{\liang{(r-1)}n}&0& \what B\what B^{\T} \\
			\end{bmatrix} \begin{bmatrix}
			\wtd \Pi^{\T} & \\ & \wtd \Pi^{\T}
			\end{bmatrix}\ltimes \begin{bmatrix}
			I\\ X\\
			\end{bmatrix} \ltimes \wtd \Pi \begin{bmatrix}
			A_F \\ \what A_F
		\end{bmatrix}
		\\&=
		\begin{multlined}[t]
			\begin{bmatrix}
				I_n &  B\what B^{\T}\ltimes X\\
				&I_{\liang{(r-1)}n}+\what B\what B^{\T}\ltimes X \\
				\end{bmatrix}\begin{bmatrix}
				A\\ \what A
				\end{bmatrix}-\begin{bmatrix}
				I_n &  B\what B^{\T}\ltimes X\\
				&I_{\liang{(r-1)}n}+\what B\what B^{\T}\ltimes X \\
				\end{bmatrix}\begin{bmatrix}
				B\\
				\what B \\
			\end{bmatrix}
			%\cdot\\\qquad
			(\what B^{\T}\ltimes X \ltimes \what B+I)^{-1}(B^{\T}X + \what B^{\T}\ltimes X\ltimes \what A)
		\end{multlined}
		\\&=\begin{bmatrix}
			A -BB^{\T}X \\
			\what A -\what BB^{\T}X \\
		\end{bmatrix}
		=\begin{bmatrix}
			A &-BB^{\T} \\
			\what A &-\what BB^{\T} \\
			\end{bmatrix}\begin{bmatrix}
			I \\ X
		\end{bmatrix}
		.
	\end{aligned}
\end{equation}
Combining \cref{eq:scare-standard-form-part1:scare,eq:scare-standard-form-part2:scare}, now \cref{eq:scare} is equivalent to
\begin{equation*}\label{eq:scare-standard-form-matrix-ltimes:scare}
	\mathcal{A}\ltimes \begin{bmatrix}
		I\\ X
	\end{bmatrix}
	= \mathcal{B} \ltimes \begin{bmatrix}
		I\\ X\\
		\end{bmatrix}\ltimes \left(\wtd \Pi\begin{bmatrix}
			A_F \\ \what A_F
	\end{bmatrix}\right)
	,
\end{equation*}
where 
\[
	\mathcal{A}=\begin{bmatrix}
		C^{\T}C & A^{\T} \\
		A &-BB^{\T} \\
		\what A &-\what BB^{\T} \\
	\end{bmatrix},\quad
	\mathcal{B} = \begin{bmatrix}
		0 & 0 & -I_n & -\what A^{\T} \\
		I_n        &0&0& B\what B^{\T} \\
		0 & I_{\liang{(r-1)}n}&0& \what B\what B^{\T} \\
		\end{bmatrix}\begin{bmatrix}
		\wtd \Pi^{\T} & \\ & \wtd \Pi^{\T}
	\end{bmatrix},
\]
which shows that the solution to the SCARE is equivalent to an invariant subspace $\range\left( \begin{bmatrix}
		I \\ X
\end{bmatrix} \right)$ of the pair $( \mathcal{A},\mathcal{B} )$ with respect to the left semi-tensor product.

%\subsection{Symplectic form}\label{ssec:symplectic-form:scare}
As continuous-time algebraic Riccati equations can be transformed to discrete-time ones by M\"obius transformation and then symplectic systems are attained,
stochastic continuous-time algebraic Riccati equations can also be transformed to stochastic discrete-time ones, which is clarified in the following.

For the M\"obius transformation, it seems that we need to consider the transformation $(\mathcal{A},\mathcal{B})\mapsto (\mathcal{A}+\gamma \mathcal{B}, \mathcal{A}-\gamma \mathcal{B})$. 
However, $\mathcal{A}, \mathcal{B}$ are not of the same size so they cannot be added directly.
Hence instead we check its equivalent effect on the invariant subspace $\range\left( \begin{bmatrix}
		I\\ X
\end{bmatrix} \right)$.
On the other hand, since in the system the part related to $\what A,\what B$ is somehow of the discrete-time style, the shifts in the M\"obius transformation are merely needed in the part related to $A,B$. 
Regarding both,
the transformation $(\mathcal{A},\mathcal{B})\mapsto (\mathcal{A}+\gamma \mathcal{B}\restrict_{\mathcal{A}}, \mathcal{A}\restrict_{\mathcal{B}}-\gamma \mathcal{B})$ is considered, where
\[
	\mathcal{B}\restrict_{\mathcal{A}}=	\begin{bmatrix}
		0    & -I_n  \\
		I_n  & 0    \\
		0    & 0    \\
	\end{bmatrix}, \qquad
	\mathcal{A}\restrict_{\mathcal{B}}= \begin{bmatrix}
		C^{\T}C & 0 & A^{\T} & 0 \\
		A        &0&-BB^{\T}& 0 \\
		\what A        &0&-\what BB^{\T}& 0 \\
		\end{bmatrix} \begin{bmatrix}
		\wtd \Pi^{\T} & \\ & \wtd \Pi^{\T}
	\end{bmatrix}
	.
\]
Note that 
\begin{align*}
	\mathcal{B}\ltimes\begin{bmatrix}
		I\\X\\
		\end{bmatrix}\ltimes \left(\wtd \Pi\begin{bmatrix}
			A_F\\\what A_F
	\end{bmatrix}\right)
	&=\mathcal{A}\ltimes \begin{bmatrix}
		I\\X
	\end{bmatrix},
	\\
	\mathcal{B}\ltimes\begin{bmatrix}
		I\\X\\
		\end{bmatrix}\ltimes \left(\wtd \Pi\begin{bmatrix}
			I_n \\ 0\\
	\end{bmatrix}\right)
	&=
	%\begin{bmatrix}
	%0 & 0 & I_n & \what A^{\T} \\
	%I_n        &0&0& -B\what B^{\T} \\
	%0 & I_{rn}&0& -\what B\what B^{\T} \\
	%\end{bmatrix}\ltimes\wtd \Pi^{\T}\ltimes \begin{bmatrix}
	%I\\ -X\\
	%\end{bmatrix}\wtd \Pi\begin{bmatrix}
	%I\\0
	%\end{bmatrix}
	%\\&=
	\begin{bmatrix}
		0 & 0 & -I_n & -\what A^{\T} \\
		I_n        &0&0& B\what B^{\T} \\
		0 & I_{\liang{(r-1)}n}&0& \what B\what B^{\T} \\
		\end{bmatrix}\begin{bmatrix}
		I\\ 0\\X\\0\\
	\end{bmatrix}
	=\mathcal{B}\restrict_{\mathcal{A}}\ltimes \begin{bmatrix}
		I\\X
	\end{bmatrix},
	\\
	\mathcal{A}\restrict_{\mathcal{B}}\ltimes\begin{bmatrix}
		I\\X\\
		\end{bmatrix}\ltimes \left(\wtd \Pi\begin{bmatrix}
			I_n\\ *
	\end{bmatrix}\right)
	&=
	%\begin{bmatrix}
	%C^{\T}C & 0 & -A^{\T} & 0 \\
	%A        &0&BB^{\T}& 0 \\
	%\what A        &0&\what BB^{\T}& 0 \\
	%\end{bmatrix}\ltimes\wtd \Pi^{\T}\ltimes \begin{bmatrix}
	%I\\ -X\\
	%\end{bmatrix}\wtd \Pi\begin{bmatrix}
	%I\\0
	%\end{bmatrix}
	%\\&=
	\begin{bmatrix}
		C^{\T}C & 0 & A^{\T} & 0 \\
		A        &0&-BB^{\T}& 0 \\
		\what A        &0&-\what BB^{\T}& 0 \\
		\end{bmatrix}\begin{bmatrix}
		I\\ *\\X\\ X\ltimes *\\
	\end{bmatrix}
	%=
	%\begin{bmatrix}
		%C^{\T}C & A^{\T} \\
		%A    & -BB^{\T} \\
		%\what A    & -\what BB^{\T} \\
		%\end{bmatrix}\begin{bmatrix}
		%I\\ X
	%\end{bmatrix}
	=\mathcal{A}\ltimes \begin{bmatrix}
		I\\X
	\end{bmatrix}
	.
\end{align*}
Hence
\begin{equation}\label{eq:ML-gamma:scare}
	\begin{aligned}[b]
		%\MoveEqLeft[0]
		(\mathcal{A}+\gamma \mathcal{B}\restrict_{\mathcal{A}})\ltimes \begin{bmatrix}
			I \\ X
		\end{bmatrix}
		%\\
		&= \mathcal{B}\ltimes \begin{bmatrix}
			I\\X
			\end{bmatrix}\ltimes \left(\wtd \Pi \begin{bmatrix}
				A_F+\gamma I\\ \what A_F
		\end{bmatrix}\right)
		\\&=\mathcal{B}\ltimes\begin{bmatrix}
			I\\X\\
			\end{bmatrix}\ltimes \wtd \Pi\left(\begin{bmatrix}
				A_F\\ \what A_F
				\end{bmatrix}(A_F+\gamma I) -\gamma \begin{bmatrix}
				A_F+\gamma I \\ 2\what A_F
		\end{bmatrix}\right)(A_F-\gamma I)^{-1}
		\\&=\left(\mathcal{A}\ltimes \begin{bmatrix}
				I\\X\\
				\end{bmatrix}(A_F+\gamma I) -\gamma \mathcal{B}\ltimes\begin{bmatrix}
				I\\X\\
				\end{bmatrix}\ltimes \wtd \Pi\begin{bmatrix}
				A_F+\gamma I \\ 2\what A_F
		\end{bmatrix}\right)(A_F-\gamma I)^{-1}
		\\&=( \mathcal{A}\restrict_{\mathcal{B}} -\gamma \mathcal{B})\ltimes\begin{bmatrix}
			I\\X\\
			\end{bmatrix}\ltimes \wtd \Pi\begin{bmatrix}
			A_F+\gamma I \\ 2\what A_F
		\end{bmatrix}(A_F-\gamma I)^{-1}
		\\&=( \mathcal{A}\restrict_{\mathcal{B}} -\gamma \mathcal{B})
		\begin{bmatrix}
			Q & \\ & Q
		\end{bmatrix}
		\begin{bmatrix}
			Q^{-1} & \\ & Q^{-1}
		\end{bmatrix}
		\ltimes\begin{bmatrix}
			I\\X\\
			\end{bmatrix}\ltimes \wtd \Pi\begin{bmatrix}
			A_F+\gamma I \\ 2\what A_F
		\end{bmatrix}(A_F-\gamma I)^{-1}
		\\&=( \mathcal{A}\restrict_{\mathcal{B}} -\gamma \mathcal{B})
		\begin{bmatrix}
			Q & \\ & Q
		\end{bmatrix}
		\ltimes\begin{bmatrix}
			I\\X\\
			\end{bmatrix}\ltimes \left(\wtd \Pi\begin{bmatrix}
				A_F+\gamma I \\ \sqrt{2\gamma}\what A_F
		\end{bmatrix}(A_F-\gamma I)^{-1}\right)
		,
	\end{aligned}
\end{equation}
where $Q:=\wtd \Pi\begin{bmatrix}
	I_n \\ & \sqrt{\frac{2}{\gamma}}I_{\liang{(r-1)}n}
	\end{bmatrix}\wtd \Pi^{\T}$ is nonsingular and $Q^{-1}=\wtd \Pi\begin{bmatrix}
	I_n \\ & \sqrt{\frac{\gamma}{2}}I_{\liang{(r-1)}n}
\end{bmatrix}\wtd \Pi^{\T}$.   
Writing
%\begin{subequations}\label{eq:ML:scare}
	\begin{alignat*}{2}
		M&=\mathcal{A}+\gamma \mathcal{B}\restrict_{\mathcal{A}}
		&&=\begin{bmatrix}
			C^{\T}C & A^{\T}-\gamma I_n
			\\
			\gamma I_n+A & -BB^{\T}
			\\
			\what A &  -\what BB^{\T}
		\end{bmatrix}
		%\label{eq:M:scare}
		,\\
		L&= ( \mathcal{A}\restrict_{\mathcal{B}} -\gamma \mathcal{B} )
		\begin{bmatrix}
			Q & \\ &\!\!\! Q
		\end{bmatrix}
		%\ltimes\wtd \Pi\begin{bmatrix}
		%I_n \\ & \sqrt{\frac{2}{\gamma}}I_{\liang{(r-1)}n}
		%\end{bmatrix}\wtd \Pi^{\T}
		&&=\begin{bmatrix}
			C^{\T}C & 0 &\!\! A^{\T}+\gamma I_n & \sqrt{2\gamma} \what A^{\T}
			\\
			A-\gamma I_n & 0 & -BB^{\T} &\!\! -\sqrt{2\gamma} B\what B^{\T}
			\\
			\what A&\!\! -\sqrt{2\gamma} I_{\liang{(r-1)}n} &-\what BB^{\T} &\!\!-\sqrt{2\gamma} \what B\what B^{\T}
			\end{bmatrix}\! \begin{bmatrix}
			\wtd \Pi^{\T} & \\ &\!\!\! \wtd \Pi^{\T}
		\end{bmatrix}\!
		,
		%\label{eq:L:scare}
	\end{alignat*}
%\end{subequations}
it can be seen that $(M,L)$ is a $\ltimes$-symplectic pair, because
$
M\ltimes J\ltimes M^{\T}
=L\ltimes J\ltimes L^{\T}$.
%SINUM%\begin{align*}
%SINUM%	\MoveEqLeft M\ltimes J\ltimes M^{\T}
%SINUM%	=L\ltimes J\ltimes L^{\T}
%SINUM%	\\&=-\begin{bmatrix}
%SINUM%		A^{\T}C^{\T}C-C^{\T}CA & C^{\T}CBB^{\T}-(\gamma I_n-A^{\T})(\gamma I_n +A^{\T}) & C^{\T}CB\what B^{\T}-(\gamma I_n-A^{\T})\what A^{\T}
%SINUM%		\\
%SINUM%		(\gamma I_n+A)(\gamma I_n-A)-BB^{\T}C^{\T}C & ABB^{\T}-BB^{\T}A^{\T}& (A+\gamma I_n)B\what B^{\T}-BB^{\T}\what A^{\T}
%SINUM%		\\
%SINUM%		\what A(\gamma I_n-A)-\what BB^{\T}C^{\T}C&\what ABB^{\T}-\what BB^{\T}(\gamma I_n+A^{\T}) & \what AB\what B^{\T}-\what BB^{\T}\what A^{\T}
%SINUM%	\end{bmatrix}
%SINUM%	.
%SINUM%\end{align*}

To apply the doubling transformation to the $\ltimes$-symplectic pair $(M,L)$, it is necessary to simplify it to a simpler form, say, $\ltimes$-SSF1 pair, 
whose existence is guaranteed by \cref{thm:initial-first-standard-form:scare}.

\begin{lemma}\label{thm:initial-first-standard-form:scare}
	Given $\gamma\ge 0$ such that
		$A_\gamma:=A-\gamma I_n$
	are nonsingular. Then $(M,L)$ is equivalent to a $\ltimes$-SSF1 pair $(\Theta_\gamma,\Phi_\gamma)$, namely there exists a nonsingular matrix $T$ %\in \mathbb{R}^{(r+2)n\times (r+2)n}$
	such that
	\begin{equation}\label{eq:Theta-Phi:scare}
			\Theta_{\gamma}=TM=\begin{bmatrix}
				E_{\gamma} & 0_{\liang{r}n\times n}
				\\
				-H_{\gamma} & I_n
			\end{bmatrix}
			_{\liang{(r+1)n\times 2n}}
			, \qquad
			\Phi_{\gamma}=TL=\begin{bmatrix}
				I_{\liang{r}n} & G_{\gamma} 
				\\
				0_{n\times \liang{r}n} & E_{\gamma}^{\T}
			\end{bmatrix}
			_{\liang{(r+1)n\times 2rn}}
			,
			%\begin{bmatrix}
			%\wtd \Pi^{\T} & \\ & \wtd \Pi^{\T}
			%\end{bmatrix},
	\end{equation}
	where 
	%where $E_{\gamma}\in \mathbb{R}^{(r+1)n\times n}, H_{\gamma}\in \mathbb{R}^{n\times n}, G_{\gamma}\in \mathbb{R}^{(r+1)n\times(r+1)n}$ satisfy $H_{\gamma}^{\T}=H_{\gamma}$ and $G_{\gamma}^{\T}=G_{\gamma}$.
	\begin{subequations}\label{eq:EGH:scare}
		\begin{alignat}{2}
			\label{eq:E:scare}
			E_\gamma &= \wtd\Pi\begin{bmatrix}
				A_\gamma+2\gamma I_n+BZ_\gamma^{\T}C
				\\
				\sqrt{2\gamma}(\what A +\what BZ_\gamma ^{\T}C )
			\end{bmatrix}
			(I_{\liang{n}}+A_\gamma^{-1}B Z_\gamma ^{\T}C)^{-1}	A_\gamma^{-1}
			&&\liang{\in \mathbb{R}^{rn\times n}}
			,\\
			\label{eq:H:scare}
			H_\gamma &=
			2\gamma A_\gamma^{-\T}C^{\T}(I_{\liang{l}}+Z_\gamma Z_\gamma^{\T})^{-1}CA_\gamma^{-1}\succeq 0
			&&\liang{\in \mathbb{R}^{n\times n}}
			, \\
			\label{eq:G:scare}
			G_\gamma &=
			\wtd \Pi\begin{bmatrix}
				\sqrt{2\gamma}  A_\gamma ^{-1}B 
				\\
\what A A_\gamma ^{-1}B-\what B
			\end{bmatrix}
			(I_{\liang{m}}+Z_\gamma^{\T}Z_\gamma)^{-1}
			\begin{bmatrix}
				\sqrt{2\gamma}  A_\gamma ^{-1}B 
				\\
\what A A_\gamma ^{-1}B-\what B
			\end{bmatrix}^{\T}
			\wtd \Pi^{\T}
			\succeq0
			&&\liang{\in \mathbb{R}^{rn\times rn}}
			.
		\end{alignat}
	\end{subequations}	
	Here $Z_\gamma = CA_\gamma^{-1}B$.
\end{lemma}
\begin{proof}
	 Directly use block elementary row transformations to obtain \cref{eq:Theta-Phi:scare}.
	In fact,  construct
	\[
		\begin{multlined}[t]
		T=
\begin{bmatrix}
		\wtd \Pi& \\ & I_n
	\end{bmatrix}
\begin{bmatrix}
		I_n & & A_\gamma ^{-1}BB^{\T}
		\\
		& I_{\liang{(r-1)}n} & \frac{1}{\sqrt{2\gamma}}K_{\gamma}B^{\T}
		\\
		& & I_n
	\end{bmatrix}\begin{bmatrix}
		I_n & & 
		\\
		& I_{\liang{(r-1)}n} &
		\\
		& & -W_{\gamma}^{-1}
	\end{bmatrix}
	\cdot\\\qquad\qquad\qquad
	\begin{bmatrix}
			I_n &&\\ & I_{\liang{(r-1)}n} & \\ -C^{\T}C & 0 & I_n
	\end{bmatrix}\begin{bmatrix}
		A_\gamma ^{-1} & 0 & 
		\\
		\frac{1}{\sqrt{2\gamma}}\what A A_\gamma ^{-1}& -\frac{1}{\sqrt{2\gamma}}I_{\liang{(r-1)}n} & 
		\\
		&&I_n
		\end{bmatrix}\*\begin{bmatrix}
		& I_{\liang{r}n} \\ I_n
	\end{bmatrix},
		\end{multlined}
\]
where $W_{\gamma}=-A_\gamma^{\T}-C^{\T}CA_\gamma^{-1}BB^{\T}=-(I_{\liang{n}}+C^{\T}Z_\gamma B^{\T}A_\gamma^{-\T})A_\gamma^{\T}$,
and $K_{\gamma}=\what A A_\gamma ^{-1}B-\what B$.
Note that 
	$(I_{\liang{n}}+C^{\T}Z_\gamma B^{\T}A_\gamma^{-\T})^{-1}
\clue{\cref{eq:smwf}}{=}
I_{\liang{n}}-C^{\T}(I_{\liang{l}}+Z_\gamma Z_\gamma^{\T})^{-1}Z_\gamma B^{\T}A_\gamma^{-\T}
$ implies $W_\gamma$ is nonsingular.
	Some calculation gives
	\begin{equation}\label{eq:pf:TM-TL}
		\begin{aligned}
					TM&=
					\begin{bmatrix}
						\wtd\Pi\begin{bmatrix}
							I_n-2\gamma W_{\gamma}^{-\T} 
							\\
							-\sqrt{2\gamma}(\what A +\what BZ_\gamma ^{\T}C )W_{\gamma}^{-\T}
						\end{bmatrix}
						&\\
					2\gamma W_{\gamma}^{-1}C^{\T}C A_\gamma ^{-1}	& I_n
					\end{bmatrix}
					,\\
					TL&=
					\begin{bmatrix}
						I_{\liang{r}n}  &-\wtd\Pi\begin{bmatrix}
							  2\gamma  A_\gamma ^{-1}BB^{\T}W_{\gamma}^{-1} & \sqrt{2\gamma}W_{\gamma}^{-\T}BK_{\gamma}^{\T}
			                  \\
							 \sqrt{2\gamma}K_{\gamma}B^{\T}W_{\gamma}^{-1} & -K_{\gamma}(I_{\liang{m}}+ Z_\gamma^{\T}Z_\gamma )^{-1}K_{\gamma}^{\T}
						\end{bmatrix}\wtd\Pi^{\T}
						\\
							    & \begin{bmatrix}
							    	I_n-2\gamma W_{\gamma}^{-1}                   & -\sqrt{2\gamma}W_{\gamma}^{-1}(\what A^{\T} +C^{\T}Z_\gamma \what B^{\T})
							    \end{bmatrix}\wtd\Pi^{\T}
					\end{bmatrix}
					.
		\end{aligned}
	\end{equation}
	Then we show \cref{eq:pf:TM-TL} is actually \cref{eq:Theta-Phi:scare}.
	For $H_\gamma$,
	\[
		\begin{aligned}[t]
					-2\gamma W_{\gamma}^{-1}C^{\T}C A_\gamma ^{-1}
			&=
			2\gamma A_\gamma^{-\T}(I_{\liang{n}}+C^{\T}Z_\gamma B^{\T}A_\gamma^{-\T})^{-1}C^{\T}CA_\gamma^{-1}
			%\\
			%&
			\clue{\cref{eq:easy}}{=} 
			2\gamma A_\gamma^{-\T}C^{\T}(I_{\liang{l}}+Z_\gamma Z_\gamma^{\T})^{-1}CA_\gamma^{-1}
			=H_\gamma;
		\end{aligned}
	\]
	for $G_\gamma$, since $B^{\T}W_\gamma^{-1}=-B^{\T}A_\gamma^{-\T}(I_{\liang{n}}+C^{\T}Z_\gamma B^{\T}A_\gamma^{-\T})^{-1}=-(I_{\liang{m}}+Z_\gamma^{\T}Z_\gamma)^{-1}B^{\T}A_\gamma^{-\T}$,
\[
	\begin{aligned}[t]
	\MoveEqLeft -\wtd\Pi\begin{bmatrix}
		2\gamma  A_\gamma ^{-1}BB^{\T}W_{\gamma}^{-1} & \sqrt{2\gamma}W_{\gamma}^{-\T}BK_{\gamma}^{\T}
		\\
		\sqrt{2\gamma}K_{\gamma}B^{\T}W_{\gamma}^{-1} & -K_{\gamma}(I_{\liang{m}}+Z_\gamma^{\T}Z_\gamma )^{-1}K_{\gamma}^{\T}
	\end{bmatrix}\wtd\Pi^{\T}
	%\\&
	=\wtd \Pi\begin{bmatrix}
			\sqrt{2\gamma}  A_\gamma ^{-1}B 
			\\
			K_{\gamma}
		\end{bmatrix}
		(I_{\liang{m}}+Z_\gamma^{\T}Z_\gamma)^{-1}
		\begin{bmatrix}
			\sqrt{2\gamma}  A_\gamma ^{-1}B 
			\\
			K_{\gamma}
		\end{bmatrix}^{\T}
		\wtd \Pi^{\T}
		=G_\gamma;
	\end{aligned}
\]
for $E_\gamma$,
\[
	\begin{aligned}[b]
		\MoveEqLeft \wtd\Pi\begin{bmatrix}
			I_n-2\gamma W_{\gamma}^{-\T} 
			\\
			-\sqrt{2\gamma}(\what A +\what BZ_\gamma ^{\T}C )W_{\gamma}^{-\T}
		\end{bmatrix}
		%\\&
		=
		\wtd\Pi\begin{bmatrix}
			2\gamma I_n-W_\gamma^{\T}
			\\
			\sqrt{2\gamma}(\what A +\what BZ_\gamma ^{\T}C )
		\end{bmatrix}
		(I_{\liang{n}}+A_\gamma^{-1}B Z_\gamma ^{\T}C)^{-1}	A_\gamma^{-1}
		=E_\gamma.
	\end{aligned}
\]
\end{proof}

Note that \cref{eq:ML-gamma:scare,eq:Theta-Phi:scare} give
	\begin{equation}\label{eq:Theta-Phi-invariant-subspace:scare}
		\Theta_{\gamma}\ltimes \begin{bmatrix}
			I_{\liang{n}} \\ X
		\end{bmatrix}
		=
		\Phi_{\gamma}\ltimes \begin{bmatrix}
			I_{\liang{n}} \\ X
		\end{bmatrix}
		\ltimes \left(\wtd \Pi\begin{bmatrix}
				A_F+\gamma I_{\liang{n}} \\ \sqrt{2\gamma}\what A_F
		\end{bmatrix}(A_F-\gamma I_{\liang{n}})^{-1}\right),
	\end{equation}
	%\begin{equation}\label{eq:continus-to-discrete-1:scare}
		%\begin{bmatrix}
			%E_{\gamma} \\ H_{\gamma}-X
		%\end{bmatrix}
		%=
		%\begin{bmatrix}
			%I & -G_{\gamma} 
			%\\
			%0 & E_{\gamma}^{\T} 
		%\end{bmatrix}
		%\begin{bmatrix}
			%I_{n+rn} \\ -X\otimes I_{r+1}
			%\end{bmatrix} \wtd \Pi \begin{bmatrix}
			%A_F+\gamma I \\ \sqrt{2\gamma}\what A_F
		%\end{bmatrix}(A_F-\gamma I)^{-1}.
	%\end{equation}
	Comparing \cref{eq:Theta-Phi-invariant-subspace:scare} with \cref{eq:sdare-matrix}, 
	similar $\ltimes$-symplectic (or detailedly $\ltimes$-SSF1) structures appear in both SCAREs and SDAREs,
	as CAREs and DAREs share similar symplectic structures.

\begin{theorem}\label{cor:scare-to-sdare:scare}
	The SCARE \cref{eq:scare-standard-form:scare} is equivalent to the following SDARE:
	\begin{equation}\label{eq:sdare:scare}
		X = E_{\gamma}^{\T} \ltimes X \ltimes 
		(I_{\liang{rn}} +G_{\gamma} \ltimes X )^{-1} \ltimes  E_{\gamma}+H_{\gamma},
	\end{equation}
	where $E_\gamma,G_\gamma,H_\gamma$ are as in \cref{thm:initial-first-standard-form:scare} for proper $\gamma>0$.
	(Here that $\gamma>0$ is proper means $A-\gamma I_n, A_F-\gamma I_n, I_{\liang{rn}}+G_\gamma\ltimes X$ are all nonsingular.)

	Moreover,
	the SDARE \cref{eq:sdare:scare} satisfies \ref{item:item1}--\ref{item:item3}, so it has a unique positive semi-definite stabilizing solution, which is also the unique stabilizing solution of the SCARE \cref{eq:scare-standard-form:scare}.
\end{theorem}
\begin{proof}
	It follows from \cref{eq:Theta-Phi-invariant-subspace:scare,eq:Theta-Phi:scare} that 
	\begin{align*}
		E_{\gamma}&=(I_{\liang{rn}}+G_{\gamma}\ltimes X )
		\wtd \Pi\begin{bmatrix}
			A_{F}+\gamma I_n \\ \sqrt{2\gamma}\what A_{F}
		\end{bmatrix}(A_F-\gamma I_n)^{-1}
		\\
		X-H_{\gamma}&=(E_{\gamma}^{\T}\ltimes X)
		\wtd \Pi\begin{bmatrix}
			A_{F}+\gamma I_n \\ \sqrt{2\gamma}\what A_{F}
		\end{bmatrix}(A_F-\gamma I_n)^{-1},
	\end{align*}
	yielding that $X-H_{\gamma} = (E_{\gamma}^{\T}\ltimes X)
	(I_{\liang{rn}}+G_{\gamma}\ltimes X)^{-1}E_{\gamma}$, which is equivalent to \cref{eq:sdare:scare}. 

	Here an issue is whether $I_{\liang{rn}}+G_\gamma\ltimes X$ is nonsingular. Note that for the solution $X$ to the SCARE, $\det(I_{\liang{rn}}+G_\gamma\ltimes X)$ is a \liang{nonzero rational function} and hence the number of $\gamma$'s to make $I_{\liang{rn}}+G_\gamma\ltimes X$ singular is finite.
	Thus there must be at least one $\gamma$ (in fact almost every real number) to meet the requirement.

	The thing left to prove is the SDARE \cref{eq:sdare:scare} has a unique positive semi-definite stabilizing solution.
%SINUM%	First,
%SINUM%	\begin{equation*}\label{eq:H-new:scare}
%SINUM%		\begin{aligned}[t]
%SINUM%			H_\gamma
%SINUM%			&=
%SINUM%			-2\gamma(-A_\gamma^{\T}-C^{\T}CA_\gamma^{-1}BB^{\T})^{-1}C^{\T}CA_\gamma^{-1}
%SINUM%			\\ &=
%SINUM%			2\gamma A_\gamma^{-\T}(I+C^{\T}CA_\gamma^{-1}BB^{\T}A_\gamma^{-\T})^{-1}C^{\T}CA_\gamma^{-1}
%SINUM%			\\
%SINUM%			&\clue{\cref{eq:easy}}{=} 
%SINUM%			2\gamma A_\gamma^{-\T}C^{\T}(I+CA_\gamma^{-1}BB^{\T}A_\gamma^{-\T}C^{\T})^{-1}CA_\gamma^{-1}
%SINUM%			\succeq 0,
%SINUM%		\end{aligned}
%SINUM%	\end{equation*}	
%SINUM%	and similarly
%SINUM%	\begin{equation*}\label{eq:G-new:scare}
%SINUM%		G_{\gamma}=\wtd \Pi\begin{bmatrix}
%SINUM%			\sqrt{2\gamma}  A_\gamma ^{-1}B 
%SINUM%			\\
%SINUM%			K_{\gamma}
%SINUM%		\end{bmatrix}
%SINUM%		(I+B^{\T}A_\gamma^{-\T}C^{\T}C A_\gamma ^{-1}B)^{-1}
%SINUM%		\begin{bmatrix}
%SINUM%			\sqrt{2\gamma}  A_\gamma ^{-1}B 
%SINUM%			\\
%SINUM%			K_{\gamma}
%SINUM%		\end{bmatrix}^{\T}
%SINUM%		\wtd \Pi^{\T}\succeq 0,
%SINUM%	\end{equation*}
%SINUM%	where $K_{\gamma}=\what A A_\gamma ^{-1}B-\what B$.
	The three matrices $E_\gamma,G_\gamma,H_\gamma$
	%writing $\Omega=CA_\gamma^{-1}B$, the three matrices 
	%\[
		%E_\gamma = \wtd\Pi\begin{bmatrix}
					%I_n-2\gamma W_{\gamma}^{-\T} 
					%\\
					%-\sqrt{2\gamma}(\what A +\what B\Omega^{\T}C )W_{\gamma}^{-\T}
				%\end{bmatrix},
				%\quad
%B_\gamma:=\wtd \Pi\begin{bmatrix}
			%\sqrt{2\gamma}  A_\gamma ^{-1}B 
			%\\
			%K_{\gamma}
		%\end{bmatrix}
		%\left(I+\Omega^{\T}\Omega\right)^{-1},
%\quad
		%C_\gamma:=\sqrt{2\gamma} \left(I+\Omega\Omega^{\T}\right)^{-1/2}CA_\gamma^{-1}
	%\]
	play the role of $A,BB^{\T},C^{\T}C$ in the SDARE \cref{eq:sdare-equivalent1}.
	Note that \ref{item:item1} holds naturally; \ref{item:item2} is guaranteed by $\N{(I_{\liang{rn}}+G_\gamma \ltimes X_\star)^{-1}\ltimes E_\gamma}<1$ for \liang{some} induced norm $\N{\cdot}$ by \cref{eq:close-loop-matrix}; 
	\ref{item:item3} is similar to \ref{item:item2}.
	Therefore, we will only show 
	\begin{equation}\label{eq:cor:proof:rho<1}
		\begin{aligned}[t]
				\N{(I_{\liang{rn}}+G_\gamma\ltimes X_\star)^{-1}\ltimes E_\gamma}
				&=\N*{\wtd \Pi\begin{bmatrix}
					A_{F}+\gamma I_n \\ \sqrt{2\gamma}\what A_{F}
			\end{bmatrix}(A_F-\gamma I_n)^{-1}}
			%\\&
			=\N*{\begin{bmatrix}
					A_{F}+\gamma I_n \\ \sqrt{2\gamma}\what A_{F}
			\end{bmatrix}(A_F-\gamma I_n)^{-1}}
			<1
		\end{aligned}
	\end{equation}
	for \liang{some} induced norm $\N{\cdot}$.

	Recall the assumption \ref{item:item2:scare}.
	Note that the adjoint of the Lyapunov operator $\op L_{\liang{\what F_X}}$ for the standard form \cref{eq:scare-standard-form:scare} is rewritten as  
\begin{equation}\label{eq:close-loop-matrix-FX:scare}
	\begin{aligned}[t]
		\op L_{\liang{\what F_X}}^{*}S
		&=(A+B\liang{\what F_X})^{\T}S+S(A+B\liang{\what F_X})+(\what A+\what B\liang{\what F_X})^{\T}\ltimes S\ltimes (\what A+\what B\liang{\what F_X}) 
		%\\&
		=A_F^{\T}S+SA_F+\what A_F^{\T}\ltimes S\ltimes \what A_F
		.
	\end{aligned}
	%\begin{aligned}
	%&\op L_{F_X}^{*}S=(A+BF_X)^{\T}S+S(A+BF_X)+(\what A+\what BF_X)^{\T}\ltimes S\ltimes (\what A+\what BF_X) 
	%\\
	%&=
	%\begin{multlined}[t]
	%\left(A-B(\what B^{\T}\ltimes X\ltimes \what B+I)^{-1}
	%(\what B^{\T}\ltimes X\ltimes \what A+B^{\T}X)\right)^{\T} S 
	%+ S\left(A-B(\what B^{\T}\ltimes X\ltimes \what B+I)^{-1}
	%(\what B^{\T}\ltimes X\ltimes \what A+B^{\T}X)\right)
	%\\
	%+\left(\what A-\what B (\what B^{\T}\ltimes X\ltimes \what B+I)^{-1}
	%(\what B^{\T}\ltimes X\ltimes \what A+B^{\T}X)\right)^{\T}\ltimes S 
	%\ltimes \left(\what A-\what B (\what B^{\T}\ltimes X\ltimes \what B+I)^{-1}
	%(\what B^{\T}\ltimes X\ltimes \what A+B^{\T}X)\right).
	%\end{multlined}
	%\end{aligned}
\end{equation}
\liang{\cite[Theorem~1.5.3]{damm2004rational}} tells the fact that \ref{item:item2:scare} is equivalent to the spectra of the Lyapunov operator $\op L_F^{*}$ \liang{being in} the interior of the left  half plane, i.e., $\rho(\op L_F^{*})\in \mathbb{C}_{-}$,
and then for $\op L_{\liang{\what F_X}}^{*}$ in \cref{eq:close-loop-matrix-FX:scare} there exists $S\succ 0$ such that $\op L_{\liang{\what F_X}}^{*} S\prec 0$.
For $\gamma>0$ to make $A_F -\gamma I$ nonsingular, %(here, $F_X$ takes \cref{eq:FX:scare})
substituting 
\begin{equation*}\label{eq:Cayley:scare}
A_F =\gamma (K-I)^{-1}(K+I) \Leftrightarrow K=(A_F +\gamma I)(A_F -\gamma I)^{-1}
\end{equation*}
into the Lyapunov operator $\op L_{\liang{\what F_X}}^{*}$ in \cref{eq:close-loop-matrix-FX:scare} gives
\[
	\gamma (K-I)^{-\T}(K+I)^{\T} S + \gamma S(K+I)(K-I)^{-1}+\what A_F ^{\T}\ltimes S\ltimes \what A_F \prec 0.
\]
By a congruent transformation, it is equivalent to %left- and right- multiplying $(K-I)^{\T}$ and $K-I$, respectively, on both sides of the above inequality show that 
\begin{align*}\label{eq:continuous-to-discrete:scare}
	0	
		&\succ \gamma (K+I)^{\T}S(K-I)+\gamma (K-I)^{\T}S(K+I)+(K-I)^{\T}\what A_F ^{\T}\ltimes S\ltimes \what A_F (K-I)
		\\
		&=2\gamma K^{\T}SK-2\gamma S + \left(\what A_F (K-I)\right)^{\T}\ltimes S\ltimes \left(\what A_F (K-I)\right)
		\\
		\intertext{by $K-I=(A_F +\gamma I)(A_F -\gamma I)^{-1}-I=2\gamma(A_F -\gamma I)^{-1}$,}
		&=2\gamma K^{\T}SK-2\gamma S +4\gamma^2 \left(\what A_F (A_F -\gamma I)^{-1}\right)^{\T}\ltimes S\ltimes\left(\what A_F (A_F -\gamma I)^{-1}\right)
		\\
		&=2\gamma\left[K^{\T}SK- S +2\gamma \left(\what A_F (A_F -\gamma I)^{-1}\right)^{\T}\ltimes S\ltimes\left(\what A_F (A_F -\gamma I)^{-1}\right)\right]
		,
\end{align*}
which implies
\[
	K^{\T}SK + 2\gamma \left(\what A_F (A_F -\gamma I)^{-1}\right)^{\T}\ltimes S\ltimes\left(\what A_F (A_F -\gamma I)^{-1}\right)\prec S.
\]
Then for $S\succ0$,
	\[
		\wtd{\op S}(S):=
		\begin{multlined}[t]
			(A_F -\gamma I)^{-\T}(A_F +\gamma I)^{\T}S(A_F +\gamma I)(A_F -\gamma I)^{-1}
			%\\
			+ 2\gamma \left(\what A_F (A_F -\gamma I)^{-1}\right)^{\T}\ltimes S\ltimes\left(\what A_F (A_F -\gamma I)^{-1}\right)
		\end{multlined}	
	\]
	is exponentially stable \liang{\cite[Theorem~2.12]{draganMS2010mathematical}}, that is, $\rho(\wtd{\op S})<1$ or \cref{eq:cor:proof:rho<1} holds.
\end{proof}

%\subsection{Fixed point iteration and doubling iteration}\label{ssec:fixed-point-iteration:scare}

Following \cref{cor:scare-to-sdare:scare} one can solve SCARE \cref{eq:scare-standard-form:scare} by any method solving the equivalent SDARE \cref{eq:sdare:scare}.
One is the fixed point iteration: 
\begin{equation*}\label{eq:fixed-point:scare}
	\begin{aligned}
		X_0& = 0,
		\qquad X_1=H_\gamma,
%SINUM%		\\
%SINUM%		X_1 &= H_{\gamma},
		\\
		X_{t+1}& %=\op D(X_t)
		= E_{\gamma}^{\T} \ltimes X_t \ltimes 
		(I_{\liang{rn}} +G_{\gamma} \ltimes X_t )^{-1} \ltimes  E_{\gamma}+H_{\gamma}.
	\end{aligned}
\end{equation*}
Another is the doubling iteration:
	\begin{subequations}\label{eq:doubling:scare}
		\begin{align}
			E_k&=E_{k-1}\ltimes (I_{\liang{r^{2^{k-1}}n}}+G_{k-1}\ltimes H_{k-1})^{-1}\ltimes E_{k-1},
			\label{eq:doubling:E:scare}
			\\
			G_{k}&=G_{k-1}\otimes I_{\liang{r}^{2^{k-1}}}+E_{k-1}\ltimes (I_{\liang{r^{2^{k-1}}n}}+G_{k-1}\ltimes H_{k-1})^{-1}\ltimes G_{k-1}\ltimes E_{k-1}^{\T},
			\label{eq:doubling:G:scare}
			\\
			H_k&=H_{k-1}+E_{k-1}^{\T}\ltimes H_{k-1}\ltimes (I_{\liang{r^{2^{k-1}}n}}+G_{k-1}\ltimes H_{k-1})^{-1}\ltimes E_{k-1},
			\label{eq:doubling:H:scare}
		\end{align}
	\end{subequations}
initially with $E_0=E_{\gamma}, G_0=G_{\gamma}, H_0=H_{\gamma}$ in \cref{eq:EGH:scare}. 

%\subsection{Doubling algorithm}\label{ssec:doubling-algorithm:scare}
Since the whole story from here on will be nearly the same as that for SDAREs, we will only briefly state the results in the following.
Besides, the properties of the fixed point iteration will also omitted, for it has been accelerated by the doubling iteration.

\begin{lemma}\label{thm:sdare-doubling:scare}
	Let $\Theta_k=\begin{bmatrix}
		E_k & 0 \\ H_k & I_{\liang{n}}
		\end{bmatrix}_{\liang{(r^{2^k}+1)n\times 2n}}$ and $\Phi_k=\begin{bmatrix}
		I_{\liang{r}^{2^k}n} & -G_k \\ 0 & E_k^{\T}
	\end{bmatrix}_{\liang{(r^{2^k}+1)n\times 2r^{2^k}n}}$. 
	Then for the doubling iteration \cref{eq:doubling:scare} with $E_0=E_{\gamma}, G_0=G_{\gamma}, H_0=H_{\gamma}$ in \cref{eq:EGH:scare}, the following statements hold:
	\begin{enumerate}
		\item \label{thm:sdare-doubling:item1:scare}
			$(\Theta_k, \Phi_k)$ is a $\ltimes$-SSF1 pair;
		\item \label{thm:sdare-doubling:item3:scare}
			$(\Theta_k,\Phi_k)\to(\Theta_{k+1}, \Phi_{k+1})=(\Theta'_k\ltimes \Theta_k, \Phi'_k\ltimes \Phi_k)$ is a $\ltimes$-doubling transformation, where
			\begin{align*}
				\Theta'_k&=\begin{bmatrix}
						E_k\ltimes (I_{\liang{r^{2^k}n}}+G_k\ltimes H_k)^{-1} & 0
						\\
						E_k^{\T}\ltimes(I_{\liang{r^{2^k}n}}+H_k\ltimes G_k)^{-1}\ltimes H_k & I_n
					\end{bmatrix}
					%_{\liang{(r^{2^k}+1)n\times 2n}}
					_{\liang{(r^{2^{k+1}}+1)n\times (r^{2^k}+1)n}}
					, \qquad 
					\\
					\Phi'_k&=\begin{bmatrix}
						I_{\liang{r}^{2^{k+1}}n} & -E_k\ltimes G_k\ltimes (I_{\liang{r^{2^k}n}}+H_k\ltimes G_k)^{-1}
						\\
						0 & E_k^{\T}\ltimes (I_{\liang{r^{2^k}n}}+H_k\ltimes G_k)^{-1}
					\end{bmatrix}
					%_{\liang{(r^{2^k}+1)n\times 2r^{2^k}n}};
					_{\liang{(r^{2^{k+1}}+1)n\times (r^{2^{k+1}}+r^{2^k})n}};
			\end{align*}
		\item \label{thm:sdare-doubling:item4:scare}
			it holds for $k=0,1,2,\dots$ that 
			\begin{equation*}\label{eq:doubling-ltimes-k:scare}
				\Theta_k\ltimes \begin{bmatrix}
					I_{\liang{n}} \\ X
				\end{bmatrix} 
				=
				\Phi_k \ltimes \begin{bmatrix}
					I_{\liang{n}} \\ X
				\end{bmatrix}
				\ltimes 
				\left(\wtd \Pi\begin{bmatrix}
						A_F+\gamma I_{\liang{n}} \\ \sqrt{2\gamma}\what A_F
				\end{bmatrix}(A_F-\gamma I_{\liang{n}})^{-1}\right)^{\ltimes 2^k}.
			\end{equation*}
	\end{enumerate}
\end{lemma}
%\begin{proof}
	%The proof is nearly the same as that of \cref{thm:sdare-doubling}, and hence omitted.
%\end{proof}

%\begin{remark}\label{rk:doubling1:scare}
For the case that $r=0$, \cref{thm:sdare-doubling:scare} degenerates into the doubling method for CAREs (see, e.g., \cite{huangLL2018structurepreserving}).
%\end{remark}

\begin{theorem}[Convergence of doubling iteration for SCAREs]\label{thm:convergence-of-doubling-iteration-for-scare}
	The sequence $\set{H_k}$  generated by the doubling iteration \cref{eq:doubling:scare} is either finite or monotonically increasing, and converges to the unique positive semi-definite  stabilizing solution $X_\star$ of the SCARE~\cref{eq:scare-standard-form:scare} R-quadratically, namely
	\begin{equation*}\label{eq:R-quadratically-convergence:scare}
		\lim_{k\to\infty}\left(\frac{\N{H_k-X_\star}}{\N{X_\star}}\right)^{1/2^k}\le
		\liang{\rho_{F_\star}}<1,
	\end{equation*}
	where $\rho_{F_\star}:=\rho\big( \left[(A_{\liang{0}}+B_{\liang{0}}F+\gamma I_{\liang{n}})\otimes(A_{\liang{0}}+B_{\liang{0}}F+\gamma I_{\liang{n}})+ 2\gamma\sum\limits_{i=1}^{\liang{r-1}}(A_i+B_iF)\otimes(A_i+B_iF)\right]\*(A_{\liang{0}}+B_{\liang{0}}F-\gamma I_{\liang{n}})^{-1}\otimes(A_{\liang{0}}+B_{\liang{0}}F-\gamma I_{\liang{n}})^{-1}  \big)$.
\end{theorem}

\section{\liang{Concluding Remarks}}\label{sec:conclusion}
In this paper we demonstrate that the stochastic AREs are essentially the deterministic AREs in the sense that all the matrix products are understood as the left semi-tensor products.
As a by-product, the fixed point iteration and the doubling iteration \liang{would play a role in acquiring the approximations to the solutions}.

\liang{However, the two iterations could not be straightforwardly used as mature numerical methods to solve the equations, 
because the left semi-tensor products make the size of involving matrices grow twice-exponentially ($r^{2^k}n$ in fact), which makes the storage an impossible task.
Take the doubling iteration \cref{eq:doubling} or \cref{eq:doubling:scare} as an example: if $n=1,r=2$, %namely only one stochastic process involved,
then the numbers of rows of first several terms $A_k$ or $E_k$ (also the number of rows/columns of $G_k$) are $2,4,16,256,65536$.
%Even if we use rank factorization on the terms $H_k=Z_kZ_k^T$ and update $Z_{k+1}=\begin{bmatrix}
	%Z_k & D_k
%\end{bmatrix}$ with calculating $D_k$ in each iteration, the total columns of $Z_{k}$ will be $1,3,15,255,65535$, which makes the storage an impossible task.
Hence more work needs to be done on developing practical algorithms, though the algebraic structure is revealed as clearly as the deterministic AREs.
}

Anyway, as we can see, many parallel theoretical results and numerical methods for DAREs and CAREs can probably be generalized to SDAREs and SCAREs. 
Plenty of results are ready to be examined,
and of course a lot of gaps are still needed to be filled.
\liang{
	We believe that there must be efficient algorithms proposed under the philosophy of this paper, and we leave it for future work.
}
%but due to the paper length, we have to stop here.
%We hope the community will be interested in enriching the story.

%\appendix
%\clearpage
\bibliographystyle{plain}
%\bibliography{../../strings,../sdals}
\bibliography{SIMAX/SAREeasy-min}

\end{document}